\documentclass[12pt]{amsart}
 
 \usepackage{graphicx,enumerate,lineno}
\numberwithin{equation}{section}
\usepackage{bbold}
\usepackage{amsmath}
\usepackage{amsthm}
\usepackage{amscd}
\usepackage{amssymb}
\usepackage{comment}

\theoremstyle{plain}
\newtheorem*{theorem*}{Theorem}
\newtheorem{theorem}[equation]{Theorem}
\newtheorem{obs}[equation]{Observation}
\newtheorem{proposition}[equation]{Proposition}
\newtheorem{lemma}[equation]{Lemma}
\newtheorem{corollary}[equation]{Corollary}

\newtheorem{question}[equation]{Question}

 \theoremstyle{definition}
\newtheorem{definition}[equation]{Definition}
\newtheorem{remark}[equation]{Remark}
 \newtheorem{example}[equation]{Example}

 \newcommand\A{{\mathcal A}}
\newcommand\C{{\mathcal C}}
 
 \newcommand\F{{\mathbb F}}
\newcommand{\f}{{\mathcal F}}

 \renewcommand\H{{\mathcal H}}
\newcommand{\M}{{\mathcal M}}
\newcommand{\N}{{\mathcal N}}
 \renewcommand{\P}{{\mathcal P}}

\newcommand{\R}{\mathbf{R}}
\newcommand{\V}{{\mathcal V}}
 \newcommand\zz{{\mathbf Z}\oplus{\mathbf Z}}
 \newcommand\Z{{\mathbf Z}}

\newcommand{\acts}{\curvearrowright}
\newcommand{\al}{\alpha}

\newcommand{\aut}{\operatorname{Aut}}
\newcommand{\be}{\beta}
\newcommand{\cat}{\operatorname{CAT}}
\newcommand{\coarse}{\infty}
\newcommand{\comm}{\operatorname{Comm}}

\newcommand{\defeq}{:=}
\newcommand{\De}{\Delta}

\newcommand\fone{{\mathbb F}^{(1)}}
\newcommand{\ga}{\gamma}
\newcommand{\Ga}{\Gamma}
\newcommand{\hd}{\operatorname{Hd}}
\newcommand{\id}{\operatorname{id}}

\newcommand{\inn}{\operatorname{Inn}}
\newcommand{\Int}{\operatorname{Int}}
\newcommand{\isom}{\operatorname{Isom}}
\newcommand{\la}{\lambda}

\newcommand{\lra}{\longrightarrow}

\newcommand{\ol}{\overline}
\newcommand{\om}{\omega}
 \newcommand{\out}{\operatorname{Out}}
\renewcommand{\parallel}{\mathbb{P}}
\newcommand\ps{{\phi_\sharp}}
\newcommand{\pt}{\operatorname{pt}}
\newcommand{\qi}{\operatorname{QI}}
\newcommand{\qp}{{\mathcal Q}}
\newcommand{\ra}{\rightarrow}
\newcommand{\restr}{\mbox{\Large \(|\)\normalsize}}

\newcommand{\Si}{\Sigma}
\newcommand{\stab}{\operatorname{Stab}}
\newcommand{\Star}{\operatorname{Star}}

\begin{document}

\title[RAAG rigidity]{The asymptotic geometry of right-angled Artin groups, I}
\author{Mladen Bestvina}
\address{Department of Mathematics\\ University of Utah\\ 155 South 1400 East, Room 233\\
 Salt Lake City, UT 84112-0090}
\email{bestvina@math.utah.edu}

\author{Bruce Kleiner}
\address{Yale University\\
Mathematics Department\\
PO Box 208283\\
New Haven, CT 06520-8283}
\thanks{M.B. is supported by an NSF grant.
B.K. is supported by NSF Grant
DMS-0505610. M.S. is supported by ISI \#580/07}
\email{bruce.kleiner@yale.edu}
\author{Michah Sageev}
\address{Department of Mathematics\\
Technion - Israel Institute of Technology\\
Haifa 32000, Israel}
\email{sageevm@techunix.technion.ac.il} 
\date{\today}
\maketitle 

\begin{abstract}
We study {\em atomic} right-angled Artin groups -- those whose defining graph has no
cycles of length $\leq 4$, and
no separating vertices, separating edges, or separating vertex stars.
We show that these groups are not quasi-isometrically rigid, but that an intermediate
form of rigidity  does hold.  
We deduce  from this that two atomic
groups are quasi-isometric iff they are isomorphic. 
\end{abstract}

 \tableofcontents

 \section{Introduction}

\subsection{Background}
We recall that to every finite simplicial graph $\Ga$, one may
associate a presentation with one generator for each vertex of $\Ga$
and one commutatation relation $[g,g']=1$ for every pair of adjacent
vertices $g,g'\in\Ga$.  The resulting group is the {\em right-angled
Artin group (RAAG) defined by $\Ga$}, and will be denoted $G(\Ga)$ (we
will often shorten this to $G$ when the defining graph $\Ga$ is
understood).  This class of groups contains the free group $F_k$,
whose defining graph has $k$ vertices and no edges, and is closed
under taking products; in particular it contains $\Z^k$ and $F_k\times
F_l$. Every RAAG
$G(\Ga)$ has a canonical Eilenberg-MacLane space $\bar K(\Ga)$ which
is a nonpositively curved cube complex (called the {\it Salvetti
  complex} in \cite{charney}); when $G$ is $2$-dimensional
$\bar K(\Ga)$ is homeomorphic to the presentation complex.  We let
$K(\Ga)$ denote the universal cover of $\bar K(\Ga)$.

RAAG's have been studied by many authors. The solution to the
isomorphism problem has the strongest form: if $G(\Gamma)\cong
G(\Gamma')$ then $\Gamma\cong\Gamma'$
\cite{MR891135}, \cite{kimetal}. Servatius \cite{MR1023285} conjectured
a finite generating set for $Aut(G(\Gamma))$ and his conjecture was
proved by Laurence \cite{laurence}. The group $G(\Gamma)$ is
commensurable with a suitable right angled Coxeter group
\cite{MR1875609}. There is an analog of Outer space in the case when
$\Gamma$ is connected and triangle-free \cite{ccv}.  For a nice
introduction to and more information about RAAG's see Charney's survey
\cite{charney}.

Our focus is on quasi-isometric rigidity properties of right-angled Artin 
groups.  Some special cases have been treated earlier: 

$\bullet$ The free group $G=F_k$.  Here the standard complex $K(\Ga)$
is a tree of valence $2k$.  Quasi-isometries are not rigid -- there
are quasi-isometries which are not at bounded distance from isometries
-- but nonetheless any finitely generated group quasi-isometric to a
free group acts geometrically on some tree
\cite{stallings}, \cite{dunwoody}, \cite[1.C1]{MR1253544} and is
commensurable to a free group \cite{MR0349850}.  Furthermore, any
quasi-action $G'\acts K$ is quasi-isometrically conjugate to an
isometric action on a tree, see \cite{moshersageevwhyte} and Section
\ref{secprelim}.

$\bullet$ $G=F_k\times F_l$.  The model space $K$ is a product of simplicial 
trees;
as with free groups, quasi-isometries are not rigid.
However, by \cite{moshersageevwhyte,kapovichkleinerleeb,ahlin,klinduced}, quasi-actions are 
quasi-isometrically conjugate to isometric actions on some product of trees.
It is a standard fact that there are groups quasi-isometric
to $G$ which are not commensurable to it (see \cite{wise,burgermozes}
for examples that are non-residually finite, or simple, respectively).

$\bullet$ $G=\Z^k$.  The model space is $\R^k$, 
which is not quasi-isometrically  
rigid.  In general, quasi-actions are not quasi-isometrically conjugate to 
isometric actions,  although this is the case for discrete cobounded 
quasi-actions \cite{polygrowth,bass}, i.e. any group quasi-isometric to 
$\Z^k$ is commensurable to it.

$\bullet$ $G(\Ga)$ where $\Ga$ is a tree of diameter at least $3$.  
Behrstock-Neumann \cite{behrstockneumann} showed that any two such Artin 
groups are quasi-isometric.
Using  work of Kapovich-Leeb \cite{kapovichleeb}, they also showed  
that a finitely generated group $G$ is quasi-isometric to such an Artin
group iff it is commensurable to one.

\subsection{Statement of results}
Our first result is that quasi-isometries of $2$-dimensional RAAG's 
preserve flats (recall that a flat is a subset isometric to $\R^2$ with
the usual metric):
\begin{theorem}
\label{thmflatspreserved}
Let $\Ga,\Ga'$ be finite, triangle-free graphs, and let 
$K=K(\Ga),\,K'=K(\Ga')$.    Then there is a constant $D=D(L,A)$
such that if $\phi:K\lra K'$ is an $(L,A)$-quasi-isometry
and $F\subset K$ is a flat, then there is a flat $F'\subset K'$
such that the Hausdorff distance satisfies
$$
\hd(\phi(F),F')<D.
$$
\end{theorem}

The remaining results in this paper concern a special class of RAAG's:

\medskip
\begin{definition}
\label{defatomic}
A finite simplicial graph $\Ga$ is {\em atomic} if 
\begin{enumerate}[(1)]
\item $\Ga$ is connected and has no vertex of valence $<2$.
\item
$\Ga$ contains no cycles of length
$<5$.
\item
$\Ga$ has no separating closed vertex stars.
\end{enumerate}
A RAAG is {\em atomic} if its defining graph is atomic.
\end{definition}

\medskip
The pentagon is the simplest example of an atomic graph.  Our main
results and most of the issues in the proofs are well illustrated by
the pentagon case.  Conditions (1)-(3) above exclude some of the known
phenomena from the examples above.  For instance, Condition (2) rules
out abelian subgroups of rank greater than $2$ and subgroups
isomorphic to $F_k\times F_l$ where $\min(k,l)>1$.  Condition (3)
prevents $G(\Ga)$ from splitting in an obvious way over a subgroup
with a nontrivial center; such a splitting would lead to a large
automorphism group.  

\begin{remark}
The main result of  \cite{laurence} implies that the
outer automorphism group of an atomic RAAG is generated by the
symmetries of $\Ga$ and inversions of generators.  It also
implies that an arbitrary  simplicial graph $\Ga$
with no cycles of length $<5$  is atomic iff 
$G(\Ga)$ has finite outer automorphism group.  
\end{remark}

The following example shows that Condition (3) is necessary in order for 
a RAAG to be determined up to isomorphism by its quasi-isometry class.

\begin{example}\label{doubled pentagon}
Let $\Ga$ be any connected graph with no triangles and choose a vertex
$v\in\Ga$. Let $f:G(\Ga)\to\Z_2$ be the homomorphism that sends $v$ to
$1\in\Z_2$ and sends all other generators to $0\in\Z_2$. Then
$G'=\ker(f)$ is a RAAG whose defining graph $\Ga'$ can be obtained from $\Ga$
by doubling along the closed star of $v\in\Ga$. Concretely, when $\Ga$
is the pentagon, $G$ and $G'$ are commensurable but not
isomorphic. Thus the atomic property for RAAG's is not a commensurability
invariant (in particular, it is not quasi-isometry invariant).
See Section \ref{secflexibility} for more discussion.
\end{example}

Until further notice, $\Ga$ will denote a fixed atomic graph, $G=G(\Ga)$,
 $\bar K=\bar K(\Ga)$ and $K=K(\Ga)$, and $V=V(K)$ will denote the vertex
set of $K$.  Whenever $\Ga'$ appears, it will
also be an atomic graph, and the associated objects will be denoted by primes.

Before stating our main rigidity theorem we need another definition:

\begin{definition}
A {\em standard circle} (respectively {\em standard torus}) is a
circle (respectively torus) in $\bar K$ associated with a vertex
(respectively edge) in $\Ga$.  A {\em standard geodesic} (respectively
{\em standard flat}) is a geodesic $\ga\subset K$ (respectively flat
$F\subset K$ ) which covers a standard circle (respectively torus).
\end{definition}
Note that if $p\in K$ is a vertex, then the  standard geodesics  
passing through $p$ are in 1-1 correspondence with    vertices in $\Ga$,
and the standard flats are in 1-1 correspondence with edges of $\Ga$.
As a consequence, the defining graph $\Ga$ is isomorphic to the 
associated incidence pattern.

\begin{theorem}[Rigidity, first version]
\label{thmmain1}
Suppose $\phi:K\ra K'$
is an $(L,A)$-quasi-isometry.  Then there is a unique bijection $\psi:V\ra V'$
of vertex sets with the following properties: 
\begin{enumerate}[(1)]
\item
$d(\psi,\phi\restr_V)<D=D(L,A)$.
\item
(Preservation of standard flats) Any two vertices $v_1,v_2\in V$ lying in a standard
flat $F\subset K$ are mapped by $\psi$ to a pair of vertices lying in a standard flat
$F'\subset K'$.
\item
(Preservation of standard geodesics)   Any two vertices $v_1,v_2\in V$ lying in a 
standard geodesic $\ga\subset K$ are mapped by $\psi$ to a pair of vertices lying on
a standard geodesic $\ga'\subset K'$.
\end{enumerate}
\end{theorem}

This theorem is proved, in the language of flat spaces, as Theorem
\ref{frigidity}. 
An immediate corollary is: 

\begin{corollary}
\label{corqiiffisom}
Atomic RAAG's are quasi-isometric iff they are isomorphic.
\end{corollary}
\noindent
This follows from Theorem \ref{thmmain1} because the pattern of standard
geodesics and standard flats passing through a vertex $p\in V$ determines
the defining graph, and by the theorem, it is preserved by quasi-isometries.

Theorem \ref{thmmain1} has two further corollaries, which may 
also be deduced from \cite{laurence}: 

\begin{corollary}[Mostow-type rigidity]
\label{cormostow}
Every isomorphism $G\ra G'$ between atomic RAAG's
is induced by a unique isometry $K\ra K'$, where we identify $G$ and $G'$
with subsets of $\isom(K)$, $\isom(K')$.  
\end{corollary}

\begin{corollary}
\label{coroutisom}
Every homotopy equivalence 
$\bar K(\Ga)\ra \bar K(\Ga)$ is homotopic to a unique isometry;
equivalently,  the homomorphism $\isom(\bar K)\ra \out(G)$
is an isomorphism.  In particular, there is an extension
$$
1\lra     H \lra \out(G)\lra \aut(\Ga)\ra 1
$$
where $\aut(\Ga)$ denotes the automorphism group of the graph $\Ga$, and $H$
consists of automorphisms that send each generator $g$ to $g^{\pm 1}$.
Thus $H$ is isomorphic to $\Z_2^V$, where $V$
is the number of vertices in $\Ga$.  
\end{corollary}

To complement the rigidity theorem, we construct examples showing that
 $K$ is not quasi-isometrically rigid, and that the failure of
 rigidity cannot be accounted for by the automorphism group, or even
 the commensurator:

\begin{theorem}
\label{thmflexibility}
Let $\comm(G)$ denote the commensurator of $G$.
Then both of the canonical homomorphisms 
$$
\aut(G)\ra\comm(G),\quad\quad \comm(G)\ra \qi(G)
$$
are injective, and have infinite index images.
\end{theorem} 

The proof is in Section \ref{secflexibility}.

Every atomic RAAG is commensurable to groups which do not admit a
geometric action on its associated $CAT(0)$ complex; see Remark
\ref{remnogeometricaction} for more discussion. In Example
\ref{doubled pentagon} the obvious involution of $G'$ cannot be
realized by an isometric involution on $K$. Nonetheless, there is 
a form of rigidity for quasi-actions of atomic RAAG's, which will
appear in a forthcoming paper:

\begin{theorem}
\label{thmquasiactions}
If $H\stackrel{\rho}{\acts} K$ is a quasi-action of a group on the standard 
complex $K$ for an
atomic RAAG, then $\rho$ is quasi-conjugate to an isometric action
$H\acts \hat X$, where $\hat X$ is a $CAT(0)$ $2$-complex.  
\end{theorem}
In fact, the $2$-complex $\hat X$ is closely related to $K$.

 The following remains open:
\begin{question}
If $G$ is an atomic RAAG and $H$ is a finitely generated group quasi-isometric to 
$G$,  are $H$ and $G$ commensurable?  Does $H$ admit a finite index
subgroup which acts isometrically on $K$?
\end{question}

\subsection{Discussion of the proofs}

\mbox{ }

The proof of Theorem \ref{thmmain1} bears some resemblance to proof 
of quasi-isometric rigidity for higher rank symmetric spaces.  One may view both
proofs as proceeding in two steps.  In the first step one shows that quasi-isometries
map top dimensional flats to top dimensional flats, up to finite Hausdorff distance; 
 in the second step one uses
the asymptotic incidence of standard flats to deduce that the quasi-isometry
has a special form.  In both proofs, the implementation of the first step 
proceeds via a structure theorem for top dimensional quasiflats, but the methods
are rather different.  The arguments used in the second step are completely different,
although in both cases they are ultimately combinatorial in nature; in the
symmetric space case one relies on Tits' theorem on building automorphisms, while
our proof requires the development and analysis of a new combinatorial object.
Another significant difference is that quasi-isometric rigidity is false in our case,
and so the rigidity statement itself is more subtle.

The proof of   Theorem \ref{thmflatspreserved}
goes roughly as follows.  We begin by invoking a result from 
\cite{quasiflats} which says that, modulo a bounded subset, every quasiflat in $K$ 
is at finite Hausdorff distance from a finite union of quarter-planes.  Here
a {\em quarter-plane} is a subcomplex isometric to a Euclidean quadrant, and 
the Hausdorff distance is controlled by the quasiflat constants.
The pattern of asymptotic incidence of quarter-planes can be encoded in the
{\em quarter-plane complex}, which is a $1$-complex analogous to the Tits
boundary of a Euclidean building or higher rank symmetric space.  
Flats in $K$ correspond to minimal length cycles in the quarter-plane complex.
Using the fact that the image of a quasiflat under a quasi-isometry is a quasiflat, 
one argues that
quasi-isometries induce isomorphisms between quarter-plane complexes.   Hence
they carry minimal length cycles to minimal length cycles, and flats to flats
(up to controlled Hausdorff distance).

The outline of the proof of Theorem \ref{thmmain1} goes as follows.

{\em Step 1. $\phi$ maps standard flats in $K$ to within uniformly bounded
Hausdorff distance of standard flats in $K'$.}   Theorem \ref{thmflatspreserved}
and standard $CAT(0)$ geometry imply that $\phi$ maps maximal product subcomplexes
to within controlled Hausdorff distance of maximal product subcomplexes.  Because
$\Ga$ is atomic, every standard flat is the intersection of two maximal product 
subcomplexes.  Since maximal product subcomplexes are preserved, their coarse 
intersections
are also preserved, and this leads to preservation of standard flats.

Before proceeding further we introduce an auxiliary object, the {\em
flat space $\F=\F(\Ga)$}, which is a locally infinite $CAT(0)$
$2$-complex associated with $\Ga$. This complex coincides with the
{\it modified Deligne complex} of Charney-Davis \cite{charneydavis}
(see also \cite{buildingsarecat0,harlandermeinert}); we
are giving it a different name since we expect that the appropriate
analog in other rigid situations will not coincide with the modified
Deligne complex. 
Start with the discrete set
$\F^{(0)}$, namely the collection of standard flats in $K$. Join two
of these points by an edge iff the corresponding flats intersect in a
standard geodesic; this defines the $1$-dimensional subset $\F^{(1)}$
of $\F$. It is convenient to think of $\F^{(0)}$ and $\F^{(1)}$ as the
$0$- and $1$-skeleton of $\F$ even though formally this is rarely the case.

To specify the rest of $\F$, we recall that the standard flats passing
through a vertex $p\in K$ correspond bijectively to the edges in
$\Ga$.  Using this correspondence, for each $p\in K$ we cone off the
corresponding subcomplex of $\F^{(1)}$ to obtain the $2$-complex $\F$.
With an appropriately chosen metric (see Section \ref{secflatspace}),
this becomes a $CAT(0)$ complex.  We use the notation $\F(p)$ to
denote the subcomplex of $\F$ corresponding to the flats passing
through $p\in K$.  It is easy to see that Theorem \ref{thmmain1} is
equivalent to saying that quasi-isometries $K\ra K'$ induce isometries
between flat spaces:

\begin{theorem}[Rigidity, second version]
\label{thmmain2}
There is a constant $D=D(L,A)$
with the following property.  If $\phi:K\ra K'$ is an $(L,A)$-quasi-isometry,
then there is a unique isometry $\phi_*:\F \ra \F'$ such that for each 
vertex $F\in \F$, the image of $F$ under  $\phi$ has  Hausdorff
distance at most $D$ from $\phi_*(F)\in\F'$.
\end{theorem}

We now switch to proving Theorem \ref{thmmain2}.  Step 1 produces a bijection
$\phi_0:\F^{(0)}\ra \F^{'(0)}$ between $0$-skeleta.

{\em Step 2. The bijection $\phi_0$ extends to an isomorphism
$\phi_*:\F\ra\F'$.}  We know that
if two distinct standard flats $F_1,\,F_2\in \F^{(0)}$ intersect in a standard
geodesic, then $\phi_0(F_1)$ and $\phi_0(F_2)$ intersect coarsely in
a geodesic, i.e. for some $D_1=D_1(L,A)$, the intersection 
$N_{D_1}(\phi_0(F_1))\cap N_{D_1}(\phi_0(F_2))$ is at controlled Hausdorff
distance from a standard geodesic in $K'$.  The property of intersecting coarsely
in a geodesic defines a relation on the collection of standard flats
which is quasi-isometry invariant.
The remainder of the argument establishes the following:

\begin{theorem}
\label{thmcombinatorial}
Any bijection
$\F^{(0)}\ra \F'^{(0)}$  which preserves the relation of intersecting
coarsely in a geodesic, is the restriction of an isometry
$\F\ra \F'$. 
\end{theorem} 
This boils down to showing that if $p\in K$, then
there is a point $p'\in K'$ such that all the flats passing through
$p$ are mapped by $\phi_0$ to flats passing through $p'$.  To
establish this we exploit {\em taut cycles}, which are special class
of cycles in the $1$-skeleton of flat space; in the pentagon case
these are just the $5$-cycles.  The heart of the proof is Theorem
\ref{TautCycleRigidity}, that the map $\phi_0$ carries the vertices of a
taut cycle to the vertices of a taut cycle.  The proof of this theorem
is based on small cancellation theory.  Due to the abundance of taut
cycles and their manner of overlap, one deduces from this that the
flats in $\F(p)$ are mapped by $\phi_0$ to $\F(p')$ for some $p'\in
K'$ (see Section \ref{secfrigidity}).

\subsection{Organization of the paper}

In Section \ref{secprelim} we discuss some background material on
RAAG's, $\cat(0)$ spaces, and cube complexes.
Sections \ref{secquarterplanes} and \ref{secpreservationofproductcomplexes}
develop the geometry of quasiflats, culminating 
in the proof of Theorem \ref{thmflatspreserved}.
Sections \ref{secflatspace}, \ref{secdualdiskdiagrams}, and 
\ref{secrigiditytautcycles} build toward the proof that (the
standard flats corresponding to) the 
vertices of a taut cycle in $\F$ are mapped by a quasi-isometry
to (the
standard flats corresponding to) the vertices of a  taut cycle (cf.
Theorem \ref{TautCycleRigidity}).
This is proved by studying cycles in $\F^{(1)}$, and certain disk fillings
of them; the argument is developed in Sections 
 \ref{secdualdiskdiagrams} and 
\ref{secrigiditytautcycles}. Section
\ref{secfrigidity} promotes taut cycle rigidity (Theorem \ref{TautCycleRigidity})
to full $\F$-rigidity (Theorem \ref{thmmain2}).
Sections \ref{proofs of corollaries}, \ref{secfurtherimplications}, and
\ref{secflexibility} work out various implications of $\F$-rigidity.

\section{Preliminaries}
\label{secprelim}

\subsection{The structure of the model space  $K(\Ga)$}
We begin by introducing some notation and terminology connected with
RAAG's.

Let $\Ga$ be a finite simplicial graph. If $\Ga$ contains no
triangles, denote by $\bar K(\Ga)$ the presentation complex for
$G(\Ga)$. Thus $\bar K(\Ga)$ has one vertex, one oriented edge for
every vertex of $\Gamma$, and one 2-cell, glued in a commutator
fashion, for every edge of $\Gamma$. The closed 2-cells are tori.
More generally, one may define
$\bar K(\Ga)$ for arbitrary $\Gamma$ by adding higher dimensional
cells (tori), one $k$-cell for every complete subgraph of $\Gamma$ on
$k$ vertices.
Then $\bar K(\Ga)$ is a nonpositively curved
complex.  The universal cover is a $CAT(0)$ cube complex denoted $K(\Ga)$.
We will often use the notation $\bar K$ or $K$, suppressing the graph
$\Ga$, when there is no risk of confusion.

We label the edges of $\bar K$ and $K$ by vertices of $\Ga$,
and the squares by edges of $\Ga$.  More generally,
each $k$-dimensional cubical face of $\bar K$ or $K$ is labelled by a $k$-tuple of vertices
in $\Ga$, or equivalently, by a
face of the flag complex of $\Ga$.  
If $Y\subset K$ is a subcomplex, we define the {\em label of $Y$}
to be the collection $\Ga_Y$ of faces of the flag complex of $\Ga$ arising
as labels of faces of $Y$.  In this paper we will be concerned primarily
with $2$-dimensional complexes, when the flag complex of $\Ga$ is $\Ga$
itself. 

Recall that a {\em full subgraph of $\Ga$} is a subgraph 
$\Ga'\subset \Ga$ such that two vertices  $v,w\in \Ga'$ 
span an edge in $\Ga'$ iff they  span an edge in $\Ga$.  (This is
closely related to the notion of the {\em induced subgraph} 
of a set of vertices of $\Ga$.)

If $\Ga'\subset \Ga$ is a full subgraph, then  there is a canonical embedding
$\bar K(\Ga')\hookrightarrow \bar K(\Ga)$,  which is locally convex and locally
isometric.  We call the image of such an embedding a {\em standard
subcomplex of $\bar K$}, or simply a {\em standard subcomplex}.  The inverse image 
of a standard subcomplex  under the
universal covering $K\ra \bar K$ is a disjoint union of convex subcomplexes,
each of which is isometric to $K(\Ga')$;  we also refer to these as
{\em standard subcomplexes}.  A {\em standard product subcomplex} is
a standard subcomplex associated with a full subgraph $\Ga'\subset \Ga$
which decomposes as a nontrivial join.  A {\em standard flat} is a 
standard subcomplex $F\subset K$ which is isometric to $\R^2$, i.e. it
is associated with a single edge in $\Ga$. 
A {\em standard geodesic} is a standard subcomplex $\ga\subset K$
associated with a single vertex in $\Ga$.
If $V\subset \Ga$ is a set of vertices in $\Ga$, then the {\em orthogonal
complement of $V$} is the set of vertices $w\in \Ga$ which
are adjacent to every element of $V$:
\begin{equation}
\label{deforthogonalcomplement}
V^\perp\defeq \{w\in\Ga\mid d(w,v)=1\;\mbox{for every}\;v\in V\}.
\end{equation}
Thus the join $V\circ V^\perp$ is a (bipartite) subgraph of $\Gamma$.
Note that the subgroup generated by $V^\perp$ lies in the centralizer
of the subgroup generated by $V$.

A {\em singular geodesic} is a geodesic $\ga\subset K$ contained in the 
$1$-skeleton of $K$.  Note that standard geodesics are singular, but
singular geodesics need not be standard.  (As an example, consider the 
case when $\Ga$ is a finite set, and $\ol{K}$ is a bouquet of circles,
and $K$ is a tree.  Then every geodesic is singular, and standard geodesics
are those which project to a single circle.)  Singular rays are defined similarly.
 A {\em quarter-plane} is a $2$-dimensional subcomplex
of $K$ isometric to a Euclidean quadrant.

For the remainder of this section, we will assume that $\Ga$ is triangle-free, 
i.e. $\dim K\leq 2$.

\begin{lemma}
\label{lemstandardalternative}
Two standard flats $F,\,F'\subset K$ lie in the parallel set $\P(\ga)$ of some geodesic
$\ga$  iff the intersection $N_r(F)\cap N_r(F')$ is unbounded for all sufficiently
large $r\in (0,\infty)$. 
\end{lemma}
\proof
Clearly, if $F\cup F'\subset \P(\ga)$ then $N_r(F)\cap N_r(F')$ contains $\ga$
when $r\geq \max(d(F,\ga),d(F',\ga))$.

By a standard argument, since the respective stabilizers in $G$ of $F$ and $F'$ 
act cocompactly on $F$ and $F'$, the stabilizer in $G$ of 
$C\defeq N_r(F)\cap N_r(F')$ acts cocompactly on $C$.  Therefore, if the convex set
$C$ is unbounded, it contains a complete geodesic $\ga$, and we have
$F\cup F'\subset \P(\ga)$.  
\qed

Note that the label $\Ga_E$ of a quarter-plane $\al\times\be=E\subset K$ 
is a bipartite graph
which is the join  $\Ga_{\al}\circ \Ga_{\be}$.

\bigskip
\begin{lemma}
\label{lemproductsarestandard}
 Suppose $C\subset K$ is a convex $2$-dimensional subcomplex which 
 splits (abstractly) as a nontrivial product of trees $C=T_1\times T_2$; here
 $T_i$ may be finite, and have no branch points.  Then $C$ is contained
 in a standard product subcomplex.  
 \end{lemma}
\proof
 $C$ is a square complex, and isomorphic to a product.  Since opposite
 edges of a square have the same label, the edge labelling of $C$
 descends to  edge labellings of $T_1$ and $T_2$.  Note if for $i\in\{1,2\}$,
 $a_i$ is an edge label appearing in $T_i$, then there is a corresponding
 square in $C=T_1\times T_2$, and hence the corresponding vertices of 
 $\Ga$ are joined by an edge.  Therefore $C$ defines a complete bipartite
 subgraph $\Ga'\subset\Ga$.  If we choose a vertex $v\in C$ and let $P$
 be the copy of $K_{\Ga'}$ passing through $v$, then clearly $C\subset P$,
 since $C$ is connected and projects to $K_{\Ga'}$.  
 \qed

\begin{lemma}
\label{lemdistanceattained}
 Suppose $C, C'$ are convex  subsets of a 
  $CAT(0)$ space  $X$, and let
 $\De\defeq d(C,C')$  (here $d(C,C')$ denote the minimum distance).  Then 
\begin{enumerate}[(1)]
\item
The sets 
 \begin{gather}
 Y\defeq \{x\in C\mid d(x,C')=\De\}\\
 Y'\defeq \{x\in C'\mid d(x,C)=\De\},
 \end{gather}
are  convex.

\item 
The nearest point map $r:X\ra C$ maps $Y'$ isometrically onto
$Y$; similarly, the nearest point map $r':X\ra C'$ maps $Y$
isometrically onto $Y'$.

\item
$Y$ and $Y'$ cobound a convex subset $Z\stackrel{\isom}{\simeq}Y\times[0,\De]$.
\item
If in addition $X$ is a locally finite $CAT(0)$ complex with cocompact
isometry group, and $C,C'$ are subcomplexes, then the sets $Y$
and $Y'$ are nonempty, and there is a constant $A>0$ such that if
$p\in C,\,p'\in C'$, and if $d(p,Y)\ge 1$, $d(p',Y')\ge 1$ then
\begin{equation}
\label{eqnlineargrowthofdistance}
d(p,C')\geq \Delta+A\,d(p,Y),\quad d(p',C)\geq \Delta+A\,d(p',Y').
\end{equation}
Furthermore, the constant $A$ depends only on $\De$ and $X$ (but not
on $C$ and $C'$).
\end{enumerate}
 \end{lemma}
\proof Assertions (1)-(3) are standard $CAT(0)$ facts, so we only
prove assertion (4).

Suppose $\{p_k\}\subset C$, $\{p_k'\}\subset C'$ are sequences such that
$d(p_k,p_k')\ra \De$.   After passing to a subsequence if necessary,
we may find a sequence $\{g_k\}\subset\isom(X)$
such that the sequences of pairs 
$$(g_k(C),g_kp_k), (g_k(C'),g_kp_k')$$
converge in the pointed Hausdorff topology to pairs
$(C_\infty,p_\infty)$, $(C_\infty',p_\infty')$.   Since there are only
finitely many subcomplexes of $X$ which are contained in a given bounded
set, we will have $p_\infty\in g_k(C)$, $p_\infty'\in g_k(C')$ for
sufficiently large $k$.  Therefore $g_k^{-1}(p_\infty)\in C$, 
$g_k^{-1}(p_\infty')\in C'$, and 
the distance $\De$ is realized.  This shows that the sets $Y_1$ and 
$Y_2$ are nonempty.  

We now prove the first estimate in (\ref{eqnlineargrowthofdistance});
the second one has a similar proof.  Observe that a convergence
argument as above implies that there is a constant $A>0$ such that if
$x\in C$ and $d(x,Y)=1$, then $d(x,C')>\Delta+A$.  If $p\in C\setminus
N_1(Y)$, then $d_{C'}$ (the distance from $C'$)   equals $\Delta$ at $r(p)$ 
and at least $\Delta+A$ at the point $x\defeq \ol{p\,r(p)}\cap S(r(p),1)$;
here $\ol{p\,r(p)}$ denotes the geodesic segment with endpoints $p$ and $r(p)$.  
Since $d_{C'}$ is convex, this implies (\ref{eqnlineargrowthofdistance}).
\qed

\begin{remark}
Note that the convex sets $Y$ and $Y'$ in 4 of Lemma \ref{lemdistanceattained}
need not be subcomplexes.
\end{remark}

\begin{corollary}
\label{corasymptoticlabels}
Suppose $\al$, $\al'$ are asymptotic singular rays in
$K$.  Then there are singular rays $\be\subset\al$, $\be'\subset \al'$
which bound a flat half-strip  subcomplex, and hence 
have the same labels.
\end{corollary}
\proof
 By passing to subrays, may assume that $\al$ and $\al'$ are subcomplexes.

Applying Lemma \ref{lemdistanceattained} to $\al$ and $\al'$,
one finds that the distance between $\al$ and $\al'$ is attained
on singular subrays $\be\subset\al$, $\be'\subset\al'$, which
bound a convex subset $Z$ isometric to $\be\times [0,\De]$.

If $\De=0$ then  $\be=\be'$ and we are done, so suppose $\De>0$.
 We may assume also that the initial point $p$ of the ray $\be$ 
is a vertex of $K$.  We claim that $Z$ is a subcomplex of $K$.  To see this, 
let  $p'\in \be'$ be the initial point  of the ray $\be'$,
 and consider the segment $\ol{p\,p'}$.  Note that there
must be a square $S\subset Z$ with vertex at $p$, whose boundary
contains the initial part of $\ol{p\,p'}$.  Repeating this
reasoning, it follows that every point $x\in\ol{p\,p'}$ is contained
in a square lying in $Z$.  Further repetition shows that $Z$ is a union
of squares, and hence is a subcomplex.

The corollary now follows from  Lemma \ref{lemproductsarestandard}.
\qed

\begin{definition}
\label{defasymptoticlabels}
Suppose $\al\subset K$ is a singular ray, and
$[\al]$ denotes its asymptote class. 
  The {\em  label of $[\al]$}
is the collection $\Ga_{[\al]}$ of labels determined by the asymptote class of 
$\al$:
$$
\Ga_{[\al]}\defeq \cap\;\{\Ga_{\al'}\mid \al'\in[\al]\}.
$$
By Corollary \ref{corasymptoticlabels}, every $\al'\in[\al]$ has a subray
$\al''\subset\al'$ such that $\Ga_{\al''}=\Ga_{[\al]}$.
If $E\subset K$ is a quarter-plane, the {\em  label of $[E]$} is the intersection
$$
\Ga_{[E]}\defeq \cap\;\{\Ga_{E'}\mid E'\in[E]\},
$$
where the two quarter-planes are equivalent if they have a quarter-plane 
in common, see Lemma \ref{lemquarterplanedistance}.
\end{definition}

\begin{lemma}
\label{lemalmostcontainedimpliescontained}
Suppose $P,P'\subset K$ are standard product complexes (in particular,
$\dim P=\dim P'=2$).  If $P\subset N_r(P')$ for some $r\in (0,\infty)$, 
then $P\subset P'$.
\end{lemma}
\proof
The distance function $d_{P'}$ is a bounded convex function on $P$.  
Since geodesic segments in $P$ are extendible in $P$, it follows that $d_{P'}$
must be constant on $P$;  set $\De\defeq d_{P'}(P)$.
By a standard $CAT(0)$ fact,  if $f:K\lra P'$ is the nearest point retraction,
then $P$ and $r(P)$ cobound a subset isometric to $P\times [0,\De]$.
Since $\dim P=\dim K$, we must have $\De=0$, and so $P\subset P'$.
\qed

\subsection{Cube complexes and hyperplanes}
\label{sechyperplane}

We recall basic terminology and facts about CAT(0) cubical complexes. For
more details, see
\cite{sageevthesis}.

A cubical complex is a
combinatorial cell complex
whose closed cells are Euclidean $n$-dimensional cubes $[0,1]^n$
of various dimensions such that the link of each vertex is a simplicial
complex (no 1-gons or 2-gons). A theorem of Gromov \cite{hypgps}  
then tells us that a simply connected cubical complex is
$\cat(0)$ if and only if the link of each vertex is a flag complex.

Since an $n$-cube is a product of $n$ unit intervals, each $n$-cube comes
equipped with $n$ natural  projection maps to the unit interval. A {\it
hypercube} is the preimage of $\{\frac12\}$ under one of these projections;
each $n$-cube contains $n$ hypercubes.
A {\it hyperplane} in a $\cat(0)$ cube complex $X$ is a  
connected subspace
intersecting each cube in a hypercube or the emptyset.
Hyperplanes are said to {\it cross} if they intersect non-trivially;
otherwise they are said to be {\it disjoint}.

Here are some basic facts about hyperplanes in $\cat(0)$ cube complexes which
we will use throughout our arguments.
\begin{itemize}
\item each hyperplane is embedded (i.e. it intersects a given cube in a
single hypercube).
\item  Each hyperplane is a track \cite{dunwoody}, and hence separates the complex into 
precisely two components, called {\it half-spaces}. 
\item if $\{H_1,\ldots,H_k\}$ is a collection of pairwise crossing
hyperplanes, then $\cap_k H_k \neq \emptyset$.
\item each hyperplane is itself a $\cat(0)$ cube complex.
\end{itemize}

\section{Quarter-planes}
\label{secquarterplanes}

In this section, we define a complex using the quarter-planes in $K$ 
and their asymptotic incidence.  This object is analogous to the 
complex at infinity that one has for symmetric spaces and Euclidean
buildings, and corresponds to part of the Tits boundary.  
The main result is Theorem \ref{thmflatcharacterization}, which 
implies that quasi-isometries between $2$-dimensional RAAG complexes
preserve flats.

Unless otherwise indicated, in this section $\Ga$ will be a triangle-free 
defining graph and $K$ will be the associated $CAT(0)$ complex.

\begin{lemma}
\label{lemquarterplanedistance}

Suppose $E=\al\times\be$ and $E'=\al'\times \be'$ are quarter-planes
in $K$.   Then one of the 
following holds:
\begin{enumerate}
\item
(Equivalent) There is a quarter-plane $E''\subset E\cap E'$; hence
$E''$ has finite Hausdorff distance from both $E$ and $E'$.
\item
(Incident) There are constants $A,B\in(0,\infty)$
such that after relabelling the factors of the $E$ and $E'$ if necessary,
$\al$ is asymptotic to $\al'$
and for every $p\in E$, $p'\in E'$, 
\begin{equation}
\label{eqnincidentestimate}
d(p,E')\geq A\left(d(p,\al)-B\right),\quad d(p',E)\geq A\left(d(p',\al')-B\right).
\end{equation}
\item
(Divergent)  The distance function $d_{E}$ grows linearly
on $E'$, and vice-versa, i.e. there are  constants $A,B\in(0,\infty)$ such that
for all $x\in E$, $y\in E'$, 
\begin{equation}
\label{eqndivergentestimate}
d(x,E')\geq A(d(x,p)-B),\quad d(y,E)\geq A(d(x,p)-B).
\end{equation}
\end{enumerate}
\end{lemma}
\proof
We apply  Lemma \ref{lemdistanceattained} 
with $C=E$ and $C'=E'$, and let $\De\in [0,\infty)$, $Y\subset E,\,Y'\subset E'$ and 
$Z\simeq Y\times [0,\De]\simeq Y'\times [0,\De]$ be as in that 
Lemma.

We claim that one of the following holds:

a. $Y=Y'$ is a quarter-plane

b. $Y$  (respectively $Y'$) is at finite Hausdorff
distance from one of the boundary rays $\al,\,\be$ (respectively $\al',\,\be'$)

c.  $Y,\,Y'$ are bounded.

To see this, first suppose  $\De=0$.  Then $Y=Y'=E\cap E'$ is a convex subcomplex 
of both $E$ and
$E'$.  Since $K$ is a square complex, this implies that $Y$ and $Y'$
are product subcomplexes of $E$, and the claim follows.
If $\De>0$, then since $Z$ meets $Y$ and $Y'$ orthogonally,
it follows that $Y,\,Y'$ are contained in the $1$-skeleton.  The claim
then follows immediately.  

 Lemma \ref{lemdistanceattained} now completes the proof.
\qed

\bigskip
\bigskip
Henceforth we will use the terms {\em equivalent}, {\em incident},
and {\em divergent}, for the $3$ cases in the lemma above.

\begin{definition}
\label{defquarterplanecomplex}
We define the {\em quarter-plane complex $\qp$} as follows.  It
has one vertex for each asymptote class of singular geodesic
rays, and one edge joining $[\al]$ to $[\be]$ for each 
equivalence class of quarter-planes $[\al\times\be]$.
\end{definition}
Note that this definition is motivated by the complex of chambers
at infinity that one has for  symmetric spaces of noncompact type,
and also by the Tits boundary.

\begin{lemma}
\label{lemquarterplaneclassification}
Suppose $E=\al\times\be\subset K$ is a quarter-plane.  Then precisely
one of the following holds:

\begin{enumerate}[(1)]
\item
$|\Ga_{[\al]}^\perp|=|\Ga_{[\be]}^\perp|=1$.   This implies
that $|\Ga_{[\al]}|=|\Ga_{[\be]}|=1$, which means that $\Ga_{[E]}$
is a single edge of $\Ga$ which forms a connected component of
$\Ga$.  The quarter-plane $E$ belongs to a unique cycle in  $\qp$, namely
 the $4$-cycle $\qp_F$ associated with a unique flat $F\subset K$ labelled
by $\Ga_{[E]}$.
\item
After relabelling the factors of $E$, $|\Ga_{[\al]}^\perp|=1$, $|\Ga_{[\be]}^\perp|>1$.
Then there is a unique equivalence class $[E']$  of quarter-planes 
such that $[E']$ is incident to $[E]$ at $[\al]$.  Furthermore,
any cycle $\Si\subset\qp$ containing $[E]$ must also contain 
$[E']$, and there is  a pair of flats $F,F'\subset K$ 
such that the corresponding $4$-cycles $\qp_F,\,\qp_{F'}\subset\qp$
intersect precisely in $[E]\cup[E']$.
\item
$\min(|\Ga_{[\al]}^\perp|,|\Ga_{[\be]}^\perp|)>1$.
Then there is a pair of flats $F,F'\subset K$ such that the 
corresponding $4$-cycles
$\qp_F,\,\qp_{F'}$ satisfy $\qp_F\cap \qp_{F'}=[E]$.
\end{enumerate}
\end{lemma}
\proof {\em Case 1.}  Since $\Ga_{[\al]}\subset \Ga_{[\be]}^\perp$ and
$\Ga_{[\be]}\subset \Ga_{[\al]}^\perp$, it follows that
$|\Ga_{[\al]}|=|\Ga_{[\be]}|=1$.  Setting $\Ga_{[\al]}=\{v\}$,
$\Ga_{[\be]}=\{w\}$, we find that $w$ is the unique vertex adjacent to
$w$, and vice-versa.  Let $K'\subset K$ be the product subcomplex
associated with the edge $\ol{vw}\subset \Ga$, containing a
sub-quarter-plane of $E$; then $K'$ is a flat which determines a cycle
$\Si_0\subset\qp$.  Note that the valence at each vertex in $\Si_0$ is
$2$, and hence any cycle $\Si\subset\qp$ which intersects $\Si_0$ must
coincide with $\Si_0$.

{\em Case 2.} Let $\Ga'\subset\Ga$ be the join of $\Ga_{[\al]}$ and
$\Ga_{[\be]}=\Ga_{[\al]}^\perp$, and let $K'\subset K$ be the copy of
$K(\Ga')$ containing a sub-quarter-plane of $E$.  If $E'$ is a
quarter-plane adjacent to $E$ along $[\al]$, then the asymptotic label
of $E'$ must be contained in $\Ga'$, and so after passing to a
sub-quarter-plane of $E'$ if necessary, we may assume that
$E'=\al'\times\be'\subset K''$, where $K''$ is another copy of
$K(\Ga')$.

We claim that $K''=K'$.  Otherwise, $K'$ and $K''$ would be disjoint, and the 
asymptotic singular rays $\al,\,\al'$ would contain subrays
$\bar\al\subset\al,\,\bar \al'\subset\al'$ such that $\bar\al$ and $\bar\al'$
bound a half-strip subcomplex $Z\subset K$.  Since $\Ga_{[\al]}^\perp=\Ga_{[\be]}\subset \Ga'$,
we have that $Z$ lies in a single copy of $K(\Ga')$, which is a contradiction.

Thus $E'\subset K'$.  But then $\be'\subset K'$ is a ray adjacent to $\al$
which is not asymptotic to $\be$; there is a unique such asymptote class
in $K'$, and this
implies that $[E']$ is unique.

Let $\Ga''\subset\Ga$ be the join of $\Ga_{[\be]}$ and $\Ga_{[\be]}^\perp\supset\Ga_{[\al]}$.
Let $K''$ be the copy of $K(\Ga'')$ which contains $K'$.  Then $K''$ is a product
$\R\times T$ where $\be$ is asymptotic to the $\R$-factor, and $T$ is a tree of 
valence $\geq 4$ because $|\Ga_{[\be]}|>1$.  The remaining assertions follow
readily from this.

{\em Case 3.}  Let $\Ga'\subset \Ga$ be the join of $\Ga_{[\al]}$ and 
$\Ga_{[\al]}^\perp$, and $\Ga''\subset\Ga$ be the join of 
$\Ga_{[\be]}$ and $\Ga_{[\be]}^\perp$.   After passing to a sub-quarter-plane
of $E$ if necessary, we may assume that $E\subset K'\cap K''$, where
$K',\,K''\subset K$ are product subcomplexes associated with $\Ga'$ and $\Ga''$
respectively.  Since $\min(|\Ga_{[\al]}^\perp|,|\Ga_{[\be]}^\perp|)>1$, there
are pairs of flats
$F_1'\cap F_2'\subset K'$ and $F_1''\cap F_2''\subset K''$ 
such that the intersections $F_1'\cap F_2', \,F_1''\cap F_2''$ are half-planes
whose intersection is precisely $E$.

\qed

In the remainder of this section, we will apply results from 
\cite{quasiflats}.  We refer the reader to that paper for the definition
and properties of support sets.

\begin{lemma}\label{quasiflats2quarterplanes}
Let  $Q\subset K$ be a quasiflat.  There is  a unique cycle 
$[E_1],\ldots,[E_k]\subset\qp$ of quarter-planes  in $\qp$,
such that the union
$$
\cup_i\;E_i
$$
has finite Hausdorff distance from $Q$.  We denote this cycle
by $\qp_Q$.
\end{lemma}
\proof
By \cite[Section 5]{quasiflats},  there is a finite collection 
$E_1,\ldots,E_k$ of quarter-planes in $K$ such that each $E_i$
is contained in the support set associated with $Q$, and for some $r\in (0,\infty)$,
$$
Q\subset N_r(\cup_i\;E_i).
$$
Note that this collection of  quarter-planes is
uniquely determined up to equivalence.  To see this, observe that  
if $E_1',\ldots,E_l'$ is another collection
of quarter-planes with 
$$
Q\subset N_{r'}(\cup_j E_j')
$$ for some $r'\in (0,\infty)$, then for each $1\leq i\leq k$, there
must be a $1\leq j\leq l$ such that $E_j'$ is equivalent to $E_i$;
otherwise Lemma \ref{lemquarterplanedistance} would imply that there
are points in $E_i$ arbitrarily far from the union $\cup_j\;E_j'$.

We now assume that the quarter-planes $E_1,\ldots,E_k$ represent
distinct equivalence classes.  

Pick $1\leq j\leq k$, and consider the quarter-plane $E_j=\al_j\times
\be_j$.  Suppose $\al_j$ is incident to $i$ of the quarter-planes in
the collection $\{E_1,\ldots,E_k\}$, where $i\neq 2$.  Pick $R\in
(0,\infty)$.  By Lemma \ref{lemquarterplanedistance}, if we choose
$x\in\al_j$ lying sufficiently far out the ray $\al_i$, then the ball
$B(x,R)$ will intersect $Q$ in a set which is uniformly
quasi-isometric to an $R$-ball $B(x',R)$ lying in an $i$-pod, where
$x'$ is a singular point of the $i$-pod. (Recall that an $i$-pod is a
tree which is a union of $i$ rays emanating from a single vertex.)
This contradicts the fact that $Q$ is a quasiflat.  Therefore
$\{E_1,\ldots,E_k\}$ determines a union of cycles in $\qp$.  But it
can only contain a single cycle, because the support set of $Q$ has
only one end.  \qed

\begin{lemma}
Every cycle $\Si\subset\qp$ arises from a quasiflat $Q\subset K$.
\end{lemma}
\proof
Let the consecutive edges in $\Si$ be represented by quarter-planes
$E_0,\ldots,E_j\subset K$, where the indices take values in the cyclic group $\Z_{j+1}$. 
 Let $W$ be the space obtained from the
disjoint union $\sqcup_i \;E_i$ by gluing the boundary rays isometrically,
in a cyclic fashion; let $\bar E_i$ denote the image of $E_i$
in $W$ under the quotient map $\pi:\sqcup\;E_i\ra W$.  With respect to the path 
metric, $W$ is biLipschitz homeomorphic to 
$\R^2$, and the quotient map $E_i\ra \bar E_i$
is a biLipschitz embedding for each $i$. Define  $\phi:W\lra K$ by setting
$\phi(w)=p$, where $p\in \cup_i\;E_i$ is a point with $\pi(p)=w$.

Since $\Si$ is a cycle, consecutive quarter-planes are incident, and this
implies that there is a constant $C\in (0,\infty)$ such that for all
$w,w'\in W$, 
$$
d(\phi(w),\phi(w'))\leq d(w,w')+C.
$$

We claim that there are constants $L,A\in (0,\infty)$ such that for all
$w,w'\in W$,
$$
d(\phi(w),\phi(w'))\geq L^{-1}d(w,w')-A.
$$
If this were false, there would be sequences $w_k,w_k'\in W$ such
that $d(w,w')\ra \infty$, and
$$
\frac{d(\phi(w_k),\phi(w'_k))}{d(w_k,w'_k)}\lra 0.
$$
By Lemma \ref{lemquarterplanedistance}, without loss of generality we may assume 
that for some $i\in \Z_{j+1}$,
 $w_k\in \Int(E_i)$ 
and $w_k'\in \Int(E_{i+1})$ for all
$k$.  Let $Y\subset E_i$, $Y'\subset E_{i+1}$ be as in Lemma \ref{lemquarterplanedistance},
where $C=E_i$, $C'=E_{i+1}$.  Then we must have
$$
\limsup_{k\ra\infty} \frac{\max(d(w_k,Y),d(w_k',Y'))}{d(w_k,w_k')}\ra 0
$$
because of (\ref{eqnincidentestimate}) or (\ref{eqndivergentestimate}).  
But then we may replace $w_k$, $w_k'$ with sequences
lying in $Y$ and $Y'$ respectively, and this yields a contradiction.
\qed

\bigskip
\begin{definition}
\label{defquasiflatcycles}
For each quasiflat $Q\subset K$, we let $\qp_Q\subset \qp$ denote
the corresponding cycle in $\qp$. 
We denote by $\C$ be the collection of cycles in $\qp$,
and by $\N$ 
the poset of subcomplexes of  $\qp$ generated by 
elements of $\C$ under finite intersection and union.  Finally, 
$\M\subset\N$ will denote the collection
of   elements of $\N$ which are minimal
among those of dimension $1$.
\end{definition}

We will consider the collection of subsets of $X$ up to Hausdorff
equivalence. Two subsets $A,B\subset X$ are {\it Hausdorff
equivalent} if for some $r>0$ we have $N_r(A)\supset B$ and
$N_r(B)\supset A$. The equivalence class of $A$ is denoted $[A]$. We
define $[A]\cup [B]$ as $[A\cup B]$. It is not hard to check that this
is independent of the choice of representatives. Note that $A\cap B$
is generally not Hausdorff equivalent to $N_r(A)\cap N_r(B)$, so
$[A]\cap [B]$ is not defined. To remedy this, say that a collection
$\{[A_i]\}$ of Hausdorff equivalence classes is {\it coherent} if for
any finite subcollection $A_{i_1},\cdots,A_{i_k}$ there is $r_0>0$
such that for every $r>r_0$ the sets $N_r(A_{i_1})\cap\cdots
N_r(A_{i_k})$ and $N_{r_0}(A_{i_1})\cap\cdots N_{r_0}(A_{i_k})$ are
Hausdorff equivalent. In this situation define $[A_{i_1}]\cap\cdots\cap
[A_{i_k}]$ as $[N_r(A_{i_1})\cap\cdots\cap N_r(A_{i_k})]$ for large
$r$. This concept behaves well under finite unions: if each $A_{i_j}$
is written as a finite union of sets, and the collection of all of
these is coherent, then the collection $A_{i_1},\cdots,A_{i_k}$ is
coherent as well.
The usual associativity and distributivity laws apply,
e.g. $([A_{i_1}]\cup [A_{i_2}])\cap [A_{i_3}]=([A_{i_1}]\cap
[A_{i_3}])\cup ([A_{i_2}]\cap [A_{i_3}])$.

For example, the collection of (the classes of) quarter-planes in $X$ is
coherent, by Lemma \ref{lemquarterplaneclassification}.   

Now consider the collection $\mathcal {QF}$ of quasiflats in $X$,
modulo Hausdorff equivalence. By Lemma
\ref{quasiflats2quarterplanes} this collection is coherent. Let
$\mathcal P$ be the collection of subsets of $X$ modulo Hausdorff
equivalence obtained by intersecting finite collections of elements of
$\mathcal {QF}$. Every element of $\mathcal P$ 
has a representative
which is a finite union of quarter-planes and standard rays, so we
have a natural map $\mathcal P\to \mathcal N$. This map preserves
finite intersections by Lemma \ref{lemquarterplaneclassification}, and
is therefore a bijection. It follows that minimal elements of
$\mathcal P$ correspond bijectively to minimal elements of
$\mathcal N$. After removing elements of $\mathcal P$ represented
by collections of rays (call those {\it inessential}, while the other
elements are {\it essential}) and elements of $\mathcal N$ that are
0-dimensional we obtain an isomorphism 
${\mathcal M}_{\mathcal P}\to {\mathcal M}$
 between the essential minimal elements of $\mathcal P$ and
1-dimensional elements of $\mathcal M$. 

If $f:X\to X'$ is a quasi-isometry, then $f$ maps quasi-flats to
quasi-flats and induces a bijection ${\mathcal M}_{\mathcal P}(X)\to 
{\mathcal M}_{\mathcal P}(X')$. This bijection preserves inessential elements, and
therefore there is an induced bijection $\mathcal M(X)\to \mathcal
M(X')$.  

\begin{comment}
Pick $p\in K$, a sequence of numbers $\{\la_i\}\subset (0,\infty)$
with $\la_i\ra 0$, and a nonprincipal ultrafilter $\om$. 
Then each quasiflat $Q\subset K$ determines a subset $Q_\om$ of the 
asymptotic cone $K_\om\defeq\om-\lim (\la_i\,K,p)$
which is bilipschitz homeomorphic to $\R^2$.  The collection of
subsets of $K_\om$ obtained this way generates a poset $\N_\om$
of subsets  under finite intersection and union;
it follows from Lemma \ref{lemquarterplanedistance} that $\N_\om$
is isomorphic to $\N$, by a dimension preserving isomorphism.
Thus the poset $\N$ (equipped with the dimension function) 
is invariant under quasi-isometries between  $2$-dimensional
RAAG complexes.
\end{comment}

Observe that Lemma \ref{lemquarterplaneclassification} yields
a classification of elements of $\M$, i.e. a  minimal
element of $\N$ consists of either $4$ edges, $2$ edges, or $1$ edge,
according to the relevant case of 
 Lemma \ref{lemquarterplaneclassification}.  Moreover, by the next lemma,
if $\Si\in \M$, then we may determine how many 
quarter-planes it contains by determining the  number of 
 elements of $\M$ that are required to complete $\Si$ to a cycle.

\begin{lemma}
Let $E=\alpha\times\beta$ be a quarter-plane as in Lemma
\ref{lemquarterplaneclassification}(2).
Then there is a
flat $F\subset K$ such that $F=l_A\times l_B$, $l_A$ is asymptotic to
$\alpha$ in the forward direction, $l_B$ is asymptotic to $\beta$ in
the forward direction, and the quarter-plane $l_A^-\times l_B^-$ formed
by the backward directions of $l_A$ and $l_B$ also satisfies that
there is a unique class of quarter-planes incident to $[l_A^-\times
  l_B^-]$ at $[l_A^-]$.
\end{lemma}

\begin{proof}
Since $\Ga_{[\alpha]}^\perp=\Ga_{[\beta]}$ consists of one label, we
may pass to a sub-quarter-plane of $E$ if necessary so that all labels
along $\beta$ are equal. By Lemma \ref{lemproductsarestandard}, $E$ is
contained in a standard product subcomplex $T_1\times T_2$. We may
replace $T_2$ by a line $l_B$ that carries the label that appears along
$\beta$. Now let $F=l_A\times l_B$ for a line $l_A$ that extends
$\alpha$ and so that all labels in $T_1$ appear infinitely often along
$l_A^-$. Then $|\Ga_{[l_B^-]}^\perp|>1$ and $|\Ga_{[l_A^-]}^\perp|=1$, so
the claim follows from Lemma \ref{lemquarterplaneclassification}(2).
\end{proof}

\begin{theorem}
\label{thmflatcharacterization}
Suppose $\Ga, \Ga'$ are finite graphs, $\Ga$ is triangle-free, and
$K=K(\Ga)$, $K'=K(\Ga')$ are the associated $CAT(0)$ square complexes.
For every $L,A\in (0,\infty)$ there is a $D\in (0,\infty)$ such that
if $f:K\ra K'$ is an $(L,A)$-quasi-isometry and $F\subset K$ is a flat,
then $\Ga'$ is triangle-free, and 
$$
\hd(f(F),F')<D
$$
for some flat $F'\subset K'$.
\end{theorem} 
\proof If $\Ga'$ contained a triangle, then $K'$ would contain a
$3$-flat $F$.  Then $F$ would contain pairs of $2$-flats $F',F''$
which are parallel but lie at arbitrarily large Hausdorff distance
from one another.  But applying a quasi-inverse of $f$ to $F',F''$, we
would obtain pairs $Q',Q''\subset K$ of quasi-flats with uniform
constants, lying at arbitrarily large --- but finite --- Hausdorff
distance from one another.  This contradicts the fact that $Q'$ and
$Q''$ lie at controlled Hausdorff distance from their support sets.
Thus $\Ga'$ must be triangle-free.

Let $\qp$ and $\qp'$ be the corresponding quarter-plane complexes,
$\C$, $\C'$ be the collections of cycles in $\qp$ and $\qp'$, and
$\N$, $\N'$ be the poset of subcomplexes of $\qp$ and $\qp'$
respectively, generated by $\C,\,\C'$ under finite intersection and
union.  As discussed above, $f$ induces a poset isomorphism
$f_*:\N\ra\N'$ which preserves dimension.  Thus it induces a bijection
between $\M$ and $\M'$ which also preserves the number of
quarter-planes.  If $\Si\subset \qp$ is a cycle, it is a union of a
uniquely determined elements of $\M$, and these are mapped by $f_*$ to
the unique elements of $\M'$ which give $f_*\Si\subset \qp'$.

If $F\subset K$ is a flat, then $\qp_F\in \C$ is a $4$-cycle,
so $\qp'_{f(F)}$ is a $4$-cycle in $\qp'$.  It follows that the 
area of the support set  $S'$ associated with $f(F)$  grows asymptotically like
$\pi r^2$.  By \cite[Section 3]{quasiflats}
this implies that $S'$ is a flat, and hence
$f(F)$ is at Hausdorff  distance at most $D=D(L,A)$ from a flat.
\qed

\section{Preservation of maximal product subcomplexes}
\label{secpreservationofproductcomplexes}

Using Theorem \ref{thmflatcharacterization}, in this section we deduce that maximal
product complexes are preserved by quasi-isometries.

Let  $\Ga$ and $\Ga'$ be triangle-free graphs, and $K,\,K'$ be the associated
$CAT(0)$ complexes.
  Let $f:K\lra K'$ be an $(L,A)$-quasi-isometry, where $\dim K\leq 2$.

\begin{lemma}
\label{lemtripodinparallelset}
There is a constant $D_0=D_0(L,A)\in (0,\infty)$ with the following property.
Suppose $Y\subset K$ is a  subcomplex isometric to the product of
 a tripod with $\R$.  Then the
singular geodesic $\ga\subset Y$ is mapped by $f$ to within Hausdorff
distance at most $D_0$ of a singular geodesic $\ga'\subset K'$, and $f(Y)$
lies in the $D_0$-neigborhood of the parallel set $\parallel(\ga')\subset K'$.
\end{lemma}
\proof
The set $Y$ is a union
$$
Y=F_1\cup F_2\cup F_3,
$$
 where the $F_i$'s are flats intersecting in $\ga$.  By 
Theorem \ref{thmflatcharacterization},
$f(F_i)$ lies at controlled Hausdorff distance from a unique flat $F_i'\subset K'$,
and hence for $r=r(L,A)\in (0,\infty)$, the intersection 
$$
W\defeq N_r(F_1')\cap N_r(F_2')\cap N_r(F_3')
$$
is quasi-isometric to $\R$.  As $W$ is convex, it contains a geodesic $\ga_1$.
Then $\parallel(\ga_1)$  contains $F_1'\cup F_2'\cup F_3'$ which implies that
$\ga_1$ is parallel to a singular geodesic $\ga'$, where the 
Hausdorff distance $\hd(\ga',\ga_1)$ is 
controlled.
\qed

\begin{theorem}
\label{thmcontainedinproductsubcomplex}
There is a constant $D=D(L,A)\in (0,\infty)$ such that if
 $P\subset K$ is a standard product subcomplex, then its
image in $K'$ is contained in the $D$-neighborhood of
a standard product subcomplex of $K'$. 
\end{theorem}
\proof
The subcomplex $P$ is associated with a subgraph $\Ga'\subset \Ga$,
where $\Ga'$ is a join $\Ga'=A\circ B$, where both $A$ and $B$ are nonempty.

{\em Case 1. $|A|=|B|=1$ and $P$ is a single flat.}  By 
Theorem \ref{thmflatcharacterization} we know that $f(P)$
lies in a neighborhood of controlled thickness around a flat
$F'\subset K'$.  But every flat in $K'$ is contained in a product subcomplex
by Lemma \ref{lemproductsarestandard}.

{\em Case 2. $|A|=1$, $|B|>1$, and $P$ is a parallel set strictly
larger than a single flat.}  Here $P\simeq T\times\R$, where $T$ is
tree of valence at least $4$ everywhere.  Let $\ga\subset P$ be a
singular geodesic of the form $\pt \times \R$.  Then each point $p\in P$ 
lies in a subset
$Y\subset P$ isometric to tripod x $\R$, where $\ga\subset Y$ is the
singular locus of $Y$.  Applying Lemma \ref{lemtripodinparallelset},
we conclude that $f(Y)$ is contained in a controlled neighborhood of
the parallel set of a singular geodesic $\ga'\subset K'$, where $\hd(\ga',f(\ga))$
is controlled.

{\em Case 3. $\min(|A|,|B|)>1$, and $P$ is a product where both
factors are strictly larger than a line.}  Let $\H,\V$ be the
collections of horizontal and vertical singular geodesics in $Y$.  By
applying Lemma \ref{lemtripodinparallelset}, we see that these map to
within controlled distance of singular geodesics in $K'$; we let
$\H'$, $\V'$ be the corresponding sets of singular geodesics.  Since
any pair $\al'\in\H',\be'\in \V'$ spans a flat plane in $K'$, their
labels must lie in a join subcomplex of the defining graph $\Ga'$.  Let
$A\subset \Ga'$ (respectively $B\subset \Ga'$) be the set of vertices
of $\Ga'$ which arise as a label of some $\al'\in\H'$ (respectively
$\be'\in\V'$).  Then $A\cup B$ spans a join subgraph $\Ga_0\subset
\Ga'$.  The collection of flats $\f'$ spanned by pairs $\al'\in \H',
\,\be'\in \V'$, must lie in a single connected component of
$p^{-1}(\bar K(\Ga_0)))\subset K'$ because if $F,F'\in\f'$, then there
is a chain of flats
$$
\{F=F_1,\ldots,F_k=F'\}\subset \f'
$$
such that $F_i\cap F_{i+1}$ contains a quarter-plane for each $1\leq i<k$.
Thus $f(P)$ lies in a controlled neighborhood of $p^{-1}(\bar K(\Ga_0)))$.
\qed

\bigskip
Note that in the situation of Theorem \ref{thmcontainedinproductsubcomplex},
the image $f(P)$ need not lie at finite Hausdorff distance from a standard
product subcomplex: consider the case when the defining graph of $K$ is the
star of a single vertex, and $P\subset K$ is a standard flat.  However,
maximal standard product subcomplexes are preserved:

\begin{corollary}\label{MaxProdSubcomplexes}
There is a constant $D_1=D_1(L,A)\in (0,\infty)$
such that if  $P\subset K$ is a maximal standard product subcomplex,
then 
$$
\hd(f(P),P')<D
$$
for some maximal standard product subcomplex $P'\subset K'$.
\end{corollary}
\proof
Let $g:K'\lra K$ be a quasi-inverse for $f$, with quasi-isometry constants
controlled by $(L,A)$.

By Theorem \ref{thmcontainedinproductsubcomplex}, we know that $f(P)\subset N_D(P')$,
where $P'\subset K'$ is a standard product complex; without loss of 
generality we may suppose that $P'$ is a maximal standard product complex.
Applying Theorem \ref{thmcontainedinproductsubcomplex} to $g$, we conclude
that $g(P')\subset N_D(P_1)$, where $P_1\subset K$ is a standard product complex.
However, this implies that $P$ lies in a finite neighborhood of $P_1$; hence
$P\subset P_1$ by Lemma \ref{lemalmostcontainedimpliescontained}, and by the maximality of $P$, we get $P=P_1$.
It follows that $P'$ lies in a controlled neighborhood of $P'$, and hence
$f(P)$ lies at controlled Hausdorff distance from a maximal standard product complex
in $K'$.
\qed

\begin{corollary}
\label{cor4cycles}
The graph $\Ga$ has $4$-cycles iff $\Ga'$ has $4$-cycles.
\end{corollary}
\proof
$\Ga$ has no $4$-cycles iff the  maximal join subgraphs of $\Ga$ are 
contained in stars
of vertices iff $K$ contains  no maximal product subcomplex quasi-isometric
to a product of two tri-valent trees.  Thus by Corollary \ref{MaxProdSubcomplexes}
the property of having $4$-cycles in the defining graph 
is a quasi-isometry invariant property of RAAG's, among
the class of RAAG's with triangle-free defining graphs.

\qed

\begin{theorem}\label{StandardFlats}
Assume that $\Gamma_1,\Gamma_2$ are connected finite graphs with all
vertices of valence $>1$ and no cycles of length $<5$. Then there is a
constant $D_2=D_2(L,A)\in (0,\infty)$ such that if $S_1\subset K_1$ is a
standard flat and $f:K_1\to K_2$ is an $(L,A)$-quasi-isometry, then
there exists a standard flat $S_2\subset K_2$ such that $f(S_1)$ and
$S_2$ are at Hausdorff distance $\le D_2$.
\end{theorem} 

In the proof we will need the following lemma. Note that maximal
product subcomplexes in $K_1,K_2$ have the form $T\times \R$ for a
4-valent infinite tree $T$, and these are precisely the parallel sets
of standard geodesics.

\begin{lemma} \label{2 parallel sets}
Let $P,P'$ be two maximal product subcomplexes in $K_1$
  (or $K_2$). Then precisely one of the following holds.
\begin{enumerate}[(1)]
\item $P=P'$,
\item $P\cap P'$ is a standard flat. Moreover, every standard flat can
  be represented as the intersection of two parallel sets.
\item There is no quarter-plane contained in Hausdorff neighborhoods
  of both $P$ and $P'$.
\end{enumerate}
\end{lemma}

\begin{proof}
If $S$ is a standard flat then $S=P\cap P'$ where $P,P'$ are parallel
sets of two perpendicular standard geodesics in $S$. Now let $P,P'$ be
two arbitrary parallel sets. Let $\Delta=d(P,P')$ and let $Y\subset P$
and $Y'\subset P'$ be as in Lemma \ref{lemdistanceattained}. If
$\Delta>0$ then $Y,Y'$ are contained in the 1-skeleta of $P,P'$ so (3)
holds. If $\Delta=0$ then $P\cap P'$ contains a vertex, say $p\in K_1$, and
we may assume that $P,P'$ are parallel sets of standard geodesics
$\ell,\ell'$ through $p$. The standard geodesics through $p$ are in
1-1 correspondence with the vertices of $\Gamma_1$. Say $\ell,\ell'$
correspond to $v,v'$. If $v=v'$ then $P=P'$ and (1) holds. If $v$ and
$v'$ are adjacent, then $P\cap P'$ is the standard flat that contains
$\ell$ and $\ell'$ (and corresponds to the edge joining $v$ to
$v'$). If $v$ and $v'$ are not adjacent then $P\cap P'=\{p\}$ and (3)
holds. 
\end{proof}

\begin{proof}[Proof of Theorem \ref{StandardFlats}]
Write $S_1=P\cap P'$ where $P,P'$ are parallel sets in $K_1$. By
Corollary \ref{MaxProdSubcomplexes}, $f(P)$ and $f(P')$ are Hausdorff
equivalent to parallel sets $Q$ and $Q'$. We must have $Q\neq Q'$
since $P\neq P'$, and the coarse intersection of $Q$ and $Q'$ is a
plane. Lemma \ref{2 parallel sets} implies that $S_2=Q\cap Q'$ is a
standard flat. By Lemma \ref{lemdistanceattained}, $S_2$ is the coarse
intersection between $Q$ and $Q'$, so it follows that $f(S_1)$ and
$S_2$ are Hausdorff equivalent.
\end{proof}

 \section{Flat space}
\label{secflatspace}

In this section we discuss two more $\cat(0)$ spaces associated with certain 
RAAG's -- an alternate model space, and flat space.
Flat space was introduced in \cite{charneydavis}, where it was called
the modified Deligne complex.   \cite{buildingsarecat0} showed that it
was a right-angled building whose apartments are modelled on the
right-angled Coxeter group $W$ with the same defining graph as the RAAG.
In this section we give a more
explicit description of the same object, and discuss some specific features
that will be needed later. 

 \subsection{Defining the exploded torus space and flat space}
 
Let $G=G(\Gamma)$ be the right angled Artin group given by a graph $\Gamma$. To
keep things simple, in the rest of the paper we will make the following
assumption on $\Gamma$:
\begin{itemize}
\item $\Gamma$ is connected, every vertex has valence $>1$, and there
  are no cycles of length $<5$.
\end{itemize}
As there are no 3-cycles, the group $G$
is 2-dimensional, and the fact that there are no 4-cycles will
guarantee that the ``flat space'' $\F$ discussed below is well-defined
and Gromov hyperbolic.

Let $\overline K=\overline K(\Gamma)$ be the presentation complex for
$G$ and let $K=K(\Gamma)$ be its universal cover. Then $K$ is a square
complex which is $\cat(0)$ thanks to our assumption that $\Gamma$ has
no 3-cycles. The complex $K$ is the standard space associated to $G$.

We will now construct a space $X=X(\Gamma)$, the {\it exploded
space}, on which $G$ acts freely and cocompactly. It is somewhat more
convenient to describe its quotient
$\overline{X}=\overline{X}(\Gamma)$, a space whose fundamental group
is $G$. Let $\Gamma^\prime$ denote the first barycentric subdivision
of $\Gamma$. We define a singular fibration $p: \overline{X}^{(1)}\to
\Gamma^\prime$ so that:

\begin{itemize}

\item the fiber over each vertex of $\Gamma^\prime$ which corresponds to the 
midpoint of an edge of $\Gamma$ is a 2-torus
\item the fiber over any other point of $\Gamma^\prime$ is a circle.

\end{itemize}
The fibration has the following local structure.  For any point
 $x\in\Gamma^\prime$ which is not the midpoint of an edge of $\Gamma$,
 the local structure is the product structure. That is, there exists a
 neighborhood $V$ of $x$ such that $p^{-1}(V)\cong V\times S^1$.  For
 a midpoint of an edge of $\Gamma$, the local structure is as follows.
 Let $A_1\cong S^1\times [0,1)$ and $A_2 \cong S^1\times [0,1)$ be two
 half open annuli. Let $T$ be a torus with two distinguished simple
 closed curves $c_1$ and $c_2$ meeting at a single point. Let $Y$ be
 the quotient space of $T\sqcup A_1\sqcup A_2$ obtained by identifying
 the boundary curve of $A_i$ with $c_i$ via a homeomorphism. Now for
 the midpoint $x$ of an edge in $\Gamma$, there exists a neighborhood
 $V$ of $x$ such that $p^{-1}(V)\cong Y$, so that $p^{-1}(x)\cong T$
 and the for any other $y\in V$, $p^{-1}(y)=S^1\times \{t\}$, a core
 circle of one of the annuli $A_i$.

\begin{figure}
\begin{center}
 \includegraphics{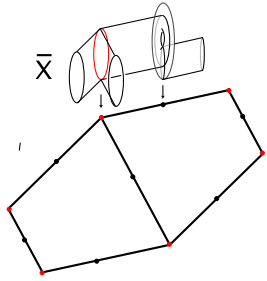}
 \end{center}
 \caption{Constructing the exploded torus space.}
 \end{figure}

Now note that since for any 2-torus, the two curves along which the
annuli are attached meet at a single point, the singular fibration
$p:\overline{X}^{(1)}\to \Gamma^\prime$ has a section $f:\Gamma^\prime\to
\overline{X}^{(1)}$. Let $C(\Gamma^\prime)$ denote the cone on
$\Gamma^\prime$. We now form the identification space
$\overline{X}\equiv(\overline{X}^{(1)}\sqcup C(\Gamma^\prime))/f$ obtained
by attaching $C(\Gamma^\prime)$ to $\overline{X}^{(1)}$ along $f$.

 By construction, we have that the fundamental group of $\overline{X}$
 is precisely the right angled artin group $G$ associated to $\Gamma$
 (simply choose the cone point $o$ of $C(\Gamma)$ as the basepoint and
 for each vertex $v$ of $\Gamma$ let $\alpha_v$ be the loop which runs
 from $o$ to $v$, along the edge $[o,v]$, around the circle associated
 to $v$, and back along $[o,v]$. Then one sees that the loops
 $\alpha_v$ generate the fundamental group and the relations are
 precisely the commutator relations dictated by $\Gamma$.)
 
We let $X$ denote the universal cover of $\overline{X}$ and let
$\rho:X\to \overline{X}$ denote the covering map. By $X^{(1)}$ denote
$\rho^{-1}(\overline X^{(1)})$. Given a torus $T$ in
$\overline{X}$ which is the fiber over the midpoint of an edge of
$\Gamma$, $\rho^{-1}(T)$ is a union of planes. Such planes are
referred to as {\it standard flats}. For each edge $e$ of
$\Gamma^\prime$, let $A_e\subset \overline{X}$ denote the annulus
which is the closure of the union of fibers lying over the interior of
$e$. A component in $X$ of $\rho^{-1}(A_e)$ is called a {\it
standard strip}.  

We may obtain $\ol{K}$ from $\ol{X}$ by collapsing each annulus to a
closed geodesic, and the cone on $\Gamma^\prime$ to a point; this defines
a surjective piecewise linear map with contractible point inverses, and hence is
a homotopy equivalence.  Likewise we get a proper homotopy equivalence
$X\ra K$ by collapsing the interval factor of each standard strip to a point,
and collapsing each copy  of $C(\Gamma^\prime)$ in $X$ to a point.

Let $\F=\F(\Gamma)$ denote the quotient of $X$ obtained by collapsing each
standard flat to a point and each standard strip $S\cong R\times I$ to
an arc by projection onto the $I$-factor. We refer to $\F=\F(\Gamma)$ as the
{\it flat space} associated to $G$. All standard flats in $X$
are represented by points in $\F$, and we can think of $\F$ as
obtained from the (discrete) set of standard flats by connecting the
dots in just the right way to obtain a very useful space (as we shall
see). 

Denote by $\F^{(1)}\subset\F$ the image of $X^{(1)}\subset X$ under the
quotient map. There is an induced action of $G$ on $\F$; let
$\pi:\F\to\overline\F$ be the quotient map. Thus $\overline\F$ can be
constructed from $\overline X$ by collapsing the tori to points and
collapsing the annuli to arcs, and therefore $\overline \F$ can be
identified with the cone $C(\Gamma)$ on $\Gamma$. We will equip
$\overline \F$ with the triangulation in which the base of the cone is
the barycentric subdivision $\Gamma'$ and $\overline X=C(\Gamma)$ is
given the cone triangulation. Note that the image of $\F^{(1)}$ in
$\overline F$ is the base of the cone. We will also equip $\F$ with
the triangulation obtained by pulling back via $\pi:\F\to\overline\F$.

There are three types of vertices
in $\F$:

\begin{itemize}
\item vertices of $\F$ which are obtained from crushing flats to
points are called {\it flat vertices} (in $\overline\F$ these
correspond to the subdivision vertices in $\Gamma'$),
\item vertices of $\F$ which are the cone points of copies of
$C(\Gamma)$ are called {\it cone vertices}, and
\item the remaining vertices are called {\it singular vertices}; these
are the vertices which correspond to the original vertices of $\Gamma$
\end{itemize}

An arc joining two flat vertices which consists of two neighboring
edges in $\F$ is called an {\it  full edge} of $\F$.

The action of $G$ on $\F$ is simplicial, the stabilizers of flat
vertices are isomorphic to $\zz$, the stabilizers of the full edges
are infinite cyclic, and $\F/G \cong C=C(\Gamma)$.  Thus this action
endows $G$ with the structure of a (simple) complex of groups, in
which the underlying complex is $C$.

\begin{comment}
\begin{remark}
The flat space $\F$ has the structure of a building. To see the
apartment structure consider the complex of groups structure just
described and replace the infinite cyclic groups by copies of $\Z_2$
and the $\zz$ groups by $\Z_2 \oplus \Z_2$. Thus the Weyl group is the
right angled Coxeter group associated with $\Gamma$.
\end{remark}
\end{comment}

We endow $X$ and $\F$ with a polyhedral metric as follows. Each torus
fiber is given the structure of a flat square 2-torus, in which each
of the designated curves along which fiber annuli are attached has
length 1. The standard annuli are viewed as quotients of the unit
square with a pair of opposite edges identified. Thus
$\overline{X}^{(1)}$ has the structure of locally $\cat(0)$ square
complex. Therefore, we just need to metrize the attached cone
$C(\Gamma^\prime)$.  Note that each 2-simplex in $C(\Gamma^\prime)$ is
of the form $(o,m,v)$, where $o$ is the cone vertex, $m$ is the
midpoint of an edge and $v$ is a vertex of $\Gamma$. We endow such a
simplex $(o,m,v)$ with the metric of an isosceles right triangle whose
legs have length 1, so that the right angle is at the vertex $v$. It
is now easy to check that the space $X$ satisfies the $\cat(0)$ condition
(i.e. the links have no loops of length less than $2\pi$). Note
further that given an edge of $\Gamma$, the cone on $\Gamma$ now
consists of precisely two isosceles right triangles, forming a unit
cube, so that in fact, $C(\Gamma^\prime)$ has the structure of a
locally $\cat(0)$ square complex.

Consequently, $X$ is a $\cat(0)$ square complex. Since $\F$ is built out
of copies of $C(\Gamma^\prime)$, we can similarly endow $\F$ with the
structure of a $\cat(0)$ square complex (for a discussion of links see
Observation \ref{obslink} below). From now on, we will abuse
notation slightly and refer to $C(\Gamma^\prime)$ simply as
$C(\Gamma)$ or simply $C$.

\begin{theorem}[No flats in flat space]
The flat space $\F$ is Gromov hyperbolic.
\end{theorem}

\begin{proof} Fix a small $\epsilon>0$.
Perturb the metric on $\F$ by replacing each Euclidean triangle
$(o,m,v)$ by a triangle with constant curvature $-\epsilon$ whose
angles are $\pi/2$ at $v$, $\pi/4$ at $m$ and $\pi/4-\epsilon/2$ at
$o$. Since $\Gamma$ has no cycles of length $<5$ it follows that this
is a $CAT(-\epsilon)$ metric on $\F$ which is quasi-isometric to the
original metric.
\end{proof}

To summarize what we have done so far, we have the following.

\begin{proposition}\label{action}
The RAAG $G$ acts on $\F$ by isometries and the following holds:
\begin{enumerate}
\item $C$ has the structure of a $\cat(0)$ square complex, with each 
square having two singular vertices, one flat  vertex and one cone vertex. 
\item The quotient space $\F/G$ is $C$. We will denote the quotient
  map $\F\to C$ by $\pi$. We lift the labeling of the vertices of $C$
  (flat, cone, singular)
  to the vertices of $\F$ via $\pi$.
\item There is a fundamental domain $\tilde{C}\subset\F$ for the action of $G$
  such that $\pi:\tilde{C}\to C$ is a homeomorphism. We will identify
  $\tilde{C}=C$. 
\item Stabilizers of flat vertices are $\Z^2$.
\item Stabilizers of singular vertices are $\Z$.
\item Stabilizers of cone vertices are trivial.
\end{enumerate}
\end{proposition}

From the point of view of complexes of groups, we can say the
following.

\begin{obs}\ %\marginpar{\tiny nix this?}

\begin{itemize}
\item $C$ is an ordered simplicial complex (after adding
  cone-to-flat diagonals in all squares), with cone vertex
  initial, and flat vertices terminal (i.e. the edges are oriented
  from cone to singular to flat).
\item $\F/G$, as an orbispace, is $C$ with the trivial label
  on the cone vertex, $\Z$ on the singular vertices, and $\Z^2$ on the
 flat vertices. 
\end{itemize}
\end{obs}

\begin{obs}\label{obslink}\ 
\begin{itemize}
\item The link in $\F$ of every cone vertex is $\Gamma$.
\item The link in $\F$ of every flat vertex is $\Z*\Z$ (the join of two 
infinite countable sets of vertices) and the
  stabilizer $\Z^2$ is acting in the obvious way: $(1,0)$ translates
  one $\Z$ by 1 and fixes the other $\Z$, and $(0,1)$ fixes the first
  $\Z$ and translates the second by 1. In particular, the complement
  of the set of all flat vertices in $\F$ is connected.
\item The link in $\F$ of every singular vertex $v$ is the join $L*\Z$ where
  $L$ is the link of the image vertex in $\Gamma$. The stabilizer $\Z$
  fixes $L$ and translates $\Z$. The link of $v$ in $\F^{(1)}$ is $L$.
\end{itemize}
\end{obs}

Each singular vertex has two natural objects associated with it: the
track and the singular star. We describe these below.

\subsection{The singular star associated to a singular vertex}

Fix a singular vertex $v\in\F$. Consider its star $S_v$ in
$\F^{(1)}$. By Observation \ref{obslink}, $S_v$ can be identified with the star of
$v$ in (barycentrically subdivided) $\Gamma$. Thus $S_v$ is the union
of all edges with one vertex at $v$ and the other vertex flat. We call
$S_v$ the {\it singular star} associated to $v$. Recall that the
parallel set $P(\gamma)$ of a geodesic line $\gamma$ in a $\cat(0)$
space is the union of all geodesic lines parallel to $\gamma$. The
vertex $v$ represents a geodesic line in $X$ and its parallel set will
be denoted by $P_v$.

\begin{obs}
Recall that $p:X\to\F$ is the quotient map.  Then
$P_v=p^{-1}(S_v)$. Conversely, every parallel set in $X$ of a standard
line (i.e. a line in a standard strip) arises in this way for a
suitable singular vertex $v$. Moreover, the parallel set is the union
of flats and strips that coarsely contain the given standard
line. Abstractly, a parallel set in $X$ is isomorphic to $T\times\R$
where $T$ is the universal cover of the 1-complex obtained from $S_v$
by wedging a circle to every vertex other than $v$.
\end{obs}

The image of a parallel set in $X$ under the quotient map $X\to\F$
will be referred to as a {\it parallel set} in $\F$ (even though it is
not the union of lines parallel to a fixed line). It is a subcomplex
of $\F^{(1)}$, and the quotient by the stabilizer is $S_v$.

\begin{obs}\label{obs cones}
The intersection between $C(\Gamma)\subset\F$ and a nontrivial
translate of it is contained in a singular star.
\end{obs}

\subsection{The track associated to a singular vertex}

Again let $v\in\F$ be a singular vertex. Let $e_1,\cdots,e_k$ be the
edges that have one endpoint at $v$ and the other endpoint flat (thus
the singular star $S_v$ is the union of the $e_i$'s). Then $C(\Gamma)$
contains $k$ squares that have $v$ as a vertex. These squares can be
enumerated $S_1,\cdots,S_k$ so that $e_i$ is one of the sides of
$S_i$, $i=1,\cdots,k$. Note that by construction, these squares share
a common edge $e$ that is incident to $v$ and whose other endpoint is
the cone point, so that $e$ is distinct from the edges
$e_1,\cdots,e_k$.  Let $\tau_v$ be the hyperplane in $C(\Gamma)$
transverse to $e$; this is well-defined since $C(\Gamma)$ is a square
complex, see Section \ref{sechyperplane}.  This is the {\it track associated to $v$}.

\begin{figure}[here]
\begin{center}
\includegraphics{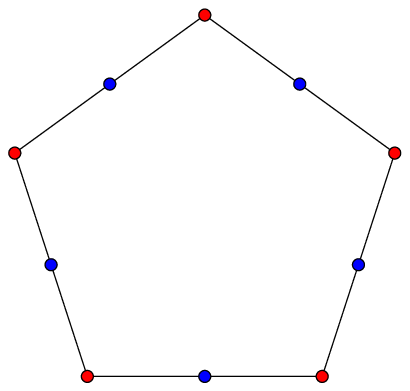}\qquad\qquad\includegraphics{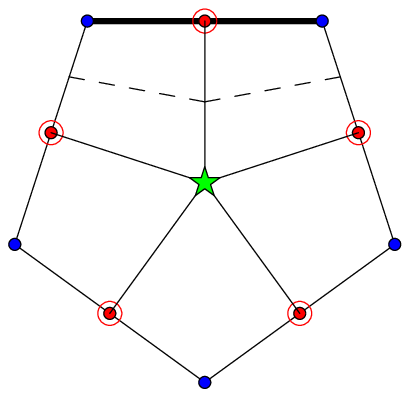}
\end{center}
\begin{caption}{$\Gamma$ is the pentagon. Also pictured is the track
    and the singular star (pictured in black) corresponding to a singular vertex}
\end{caption}
\end{figure}

\begin{figure}[here]
\centerline{\includegraphics{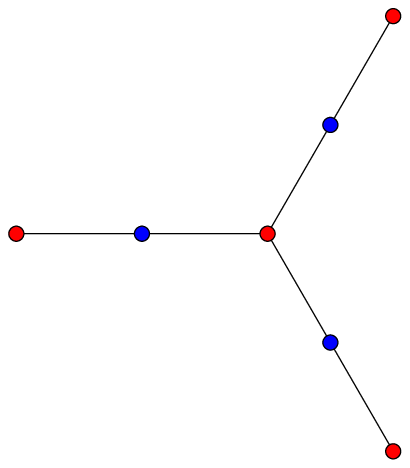}\qquad\qquad\includegraphics{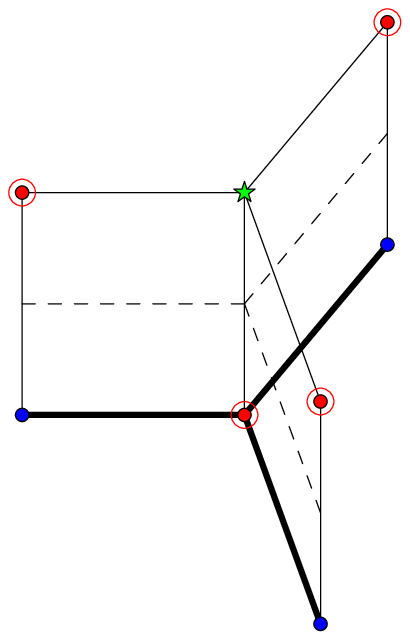}}
\begin{caption}{$\Gamma$ is the tripod. Also pictured is the track
    and the singular star corresponding to the central singular vertex}
\end{caption}
\end{figure}

\begin{obs}\label{tracks}\ 
Recall that $\pi:\F\to C(\Gamma)$ is the quotient map.
\begin{itemize}
\item $\pi^{-1}(\tau_v)$ is a collection of hyperplanes in $\F$. Each
  hyperplane is a convex subset of $\F$.
\item Conversely, any hyperplane in $\F$ is a component of
  $\pi^{-1}(\tau_v)$ for a suitable singular vertex $v$.
\end{itemize}
\end{obs}

Also denote by $R_i$ the closure of the
  component of $S_i-\tau_v$ that contains
  $e_i$. Define the {\it thickened track associated to $v$} to be 
$$R_v:=R_1\cup\cdots\cup R_k$$

\begin{obs}\label{product}\ 
\begin{itemize}
\item
$R_v$ is a product neighborhood of $P_v$, that is, $(R_v,P_v)\cong
  (P_v\times I,P_v\times\{0\})$.
\item
Let $\tilde R_v$ be a component of $\pi^{-1}(R_v)$. Then $\tilde R_v$
  contains a unique component $\tilde P_v$ of
  $\pi^{-1}(P_v)$. Moreover, $(\tilde R_v,\tilde P_v)\cong (\tilde
  P_v\times c(\Z),\tilde P_v\times\{c\})$, where $c(\Z)$ is the cone
  on $\Z$ with cone point $c$.
\item
Under this identification, the pointwise stabilizer $\Z$ of $\tilde
P_v$ acts on $\tilde R_v\cong \tilde
  P_v\times c(\Z)$ by fixing $\tilde P_v$ and the cone point,
  and translating the
  base of the cone.
\end{itemize}
\end{obs}

\begin{lemma}\label{parallel sets}\  
\begin{itemize}
\item
Every parallel set in $\F$ separates $\F$. The set of complementary
components is acted on freely and transitively by the $\Z$ subgroup
fixing  the parallel set pointwise.
\item For any two distinct cone vertices in $\F$ there is a parallel
  set that separates them.
\end{itemize}
\end{lemma}
\proof
A parallel set
$P_v\subset \F$ has a product neighborhood $P_v\times c(\Z)$ as described in 
Observation \ref{product}.  It follows from the simple connectivity of $\F$
that the connected components of $\F\setminus P_v$ are in bijective
correspondence with the connected components of $(P_v\times c(\Z))\setminus P_v$;
the latter are in bijective correspondence with the $\Z$ subgroup fixing
$P_v$ pointwise.  
For the second assertion, consider the geodesic joining the two
vertices. This geodesic must cross some hyperplane, and hence some
component of some $\pi^{-1}(\tau_v)$ (see Observation
\ref{tracks}). The associated parallel set separates between the two
cone vertices (see Observation \ref{product}).
\qed

\subsection{Coarse distance and parallel sets}

We will be interested in paths and loops in $\F^{(1)}$; these will
correspond to sequences of standard flats in $X$. There is a
straighforward notion of distance between standard flats. Suppose that
$F$ and $F^\prime$ are standard flats. The flats $F$ and $F^\prime$
correspond to two vertices in $\F$.  Note that since $\F^{(1)}$
is a bipartite graph, the usual path distance in $\F^{(1)}$ between
 $F,\,F^\prime$ is even;  let $D(F,F^\prime)$ denote half this
path  distance in $\F^{(1)}$.  
We will be interested in a
somewhat different notion of distance on $\F^{(1)}$, called {\it coarse
distance}, which we now describe.
 We say that the coarse distance
between $F$ and $F^\prime$ is 1, $D^\infty(F,F^\prime)=1$, if they
coarsely intersect in a line, or equivalently, if belong to the same parallel
set (cf. Lemma \ref{lemstandardalternative}). We will say that $D^\infty(F,F^\prime)=n$, 
if their exists a
sequence of standard flats $F=F_0,...,F_n=F^\prime$ such that
$D^\infty(F_i,F_{i+1})=1$ and $n$ is the smallest number for which
there exists such a sequence. Note that this then defines a metric on
the set of flat vertices of $\F$. The reason this metric is natural is
that it is preserved by quasi-isometries: the binary relation on $\F^{(0)}$
of belonging to the same parallel set is quasi-isometry invariant,
by Lemma \ref{lemstandardalternative} and Theorem \ref{StandardFlats}.

The metric $D^\infty$ can be seen directly in $\F$ as follows. Suppose
that $e$ and $f$ are two full edges of $\F$ meeting at a flat vertex
$v$. We say that $e$ and $f$ define a {\it legal turn at $v$} if $e$
and $f$ have cyclic stabilizers with trivial intersection. Otherwise,
they have the same cyclic stabilizer and we say that they define an
{\it illegal turn at $v$}.  Now suppose that $v$ and $w$ are two flat
vertices of $\F$. Suppose that $\alpha$ is a full edge path in $\F$
joining $v$ and $w$. Then the {\it coarse length} of $\alpha$,
$length_\coarse(\alpha)$ is computed by counting the number of legal
turns along $\alpha$ and adding 1. Thus, the coarse distance between
vertices is simply the minimal coarse length of a path between them.

If a path in $\F^{(1)}$ has no legal turns in it, then in fact all the
flat vertices along it correspond to flats which are all contained
in the same parallel set. We call such a path a {\it stalling path}. The union of
all stalling paths containing a given singular vertex is thus a
parallel set in $\F$; it corresponds in $X$ to the union of all
flats and strips coarsely containing a given singular line.

%
%SECTION:DUAL DISK DIAGRAMS
%

\section{Dual disk diagrams}
\label{secdualdiskdiagrams}

We consider $\F$ as a $\cat(0)$ square complex. Let $\H$ denote the union
of hyperplanes in $\F$. We call a component of $\F-\H$ a {\it
block}. Notice that each block contains a unique vertex of $\F$.
Suppose that $\alpha$ is a closed full edge path in $\fone$. Then
$\alpha:S^1 \to \fone$ extends to a map $\Delta: D^2\to \F$. Now we
may make $\Delta$ transverse to the hyperplanes of $\F$. Thus
$\A=\Delta^{-1}(\H)$ is a union of embedded arcs and simple closed curves in
$D=D^2$, where each arc and closed curve is a component of the preimage of a single hyperplane. 
As in \cite{sageevthesis}, one may assume that there are in
fact no simple closed curves and that each pair of arcs meets in at
most a single point. Since $\F$ is a 2-complex, we may further assume
that there are no triangular regions: every collection of three arcs
contains a disjoint pair. The pair $(D,\A)$ is called a {\it dual disk
diagram} for $\alpha$. Sometimes we will abuse notation and just refer
to $D$ as the dual disk diagram. In \cite{sageevthesis} it is shown that one can get
from any diagram to any other through a sequence of triangular
moves. Since there are no triangles, we see that the dual disk diagram
is unique.

A {\it region} of $D$ refers to the closure of a component of $D -
(\partial D \cup \A)$. Note that each region $D$ is associated to a
vertex of $X$, namely the vertex whose block the region is mapped
to. By a {\it boundary region} we mean a region that intersects
$\partial D$. Regions that are not boundary regions are called {\it internal} regions.  The union 
of all internal regions is called the {\it
core} of the diagram.  By a {\it corner} of $D$ we mean a triangular
region bounded by two arcs of $\A$ and a subarc of $\partial D$. A
corner corresponds to $\alpha$ running around the corner of a square
of $\F$. Since $\alpha$ is a path of full edges, this means that the
corner of the square is a flat vertex and the turn at the vertex is a
legal turn.

We now need to discuss some complications which can occur in general disk diagrams, 
but which will not appear in our setting, because our boundary
map $\al$ may be assumed to be an embedding.  

\subsection{Spurs} 
A {\it spur} of $D$ consists of a collection of nested arcs, each disjoint from all other arcs, so that
one of the  arcs has endpoints on neighboring boundary edges (this arc bounds a region with no curves or arcs in it) ; see
figure below.

In this case, we see that the original full edge path has some
backtracking. Since we will always be dealing with paths that have no
backtracking, we can assume that there are no spurs.

\begin{center}
\includegraphics{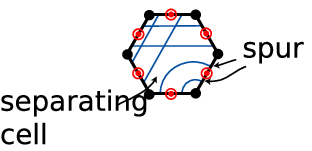}
\end{center}

\subsection{Separating cells} 
Suppose that  the intersection $P\cap \partial D$ 
of a closed cell $P$
with the boundary of  of $D$ has two distinct connected arc
components $\ga,\,\ga'$, so that $P$ separates $D$. Then the path 
represented by $\partial D$ is not embedded. To see this, note that
since $\partial D$ is a  path of full edges, $\ga$ and $\ga'$ must each
contain a vertex, and these vertices must map to the same vertex
in $\F$.    Because we will be dealing with embedded paths, we can
assume that there are no separating cells.

The upshot of the above is that if the edge loop is embedded, then the
disk diagram associated to it has no spurs or separating cells. An
embedded full edge loop will be called a {\it cycle}.

Recall that there are three types of vertices: singular, flat, and
cone vertices. Since each region of $D$ is mapped to a block, which is
uniquely associated to a vertex of $\F$, we have three types of
regions, which we also call singular, flat and cone regions.

\begin{obs}\label{obs1} Every square in $\F$ has two diagonally opposite singular
  vertices, one cone vertex and one flat vertex. 
\end{obs}

\begin{center}
\includegraphics{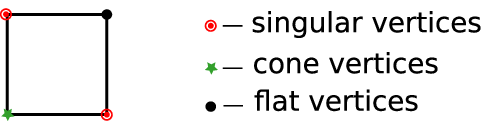}
\end{center}

Below is an example of a dual disk diagram. The yellow cells are the
boundary cells and the remaining cells form the core.

\begin{center} 
\includegraphics{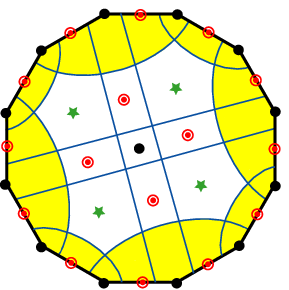}
\end{center}

\begin{obs}\label{obs3}
Since corners cannot contain singular vertices, all corner regions are
flat regions. Also, all cone regions are interior.
\end{obs}

\begin{obs}\label{obs2} 
Observation \ref{obs1} can be transported to the disk diagram to tell
us what types the regions are; around each intersection point of two
arcs in the diagram, we have two diagonally opposite singular regions,
one cone region and one flat region. In particular, this tells us that
(under our standing assumption that our paths are embedded) the core
is always not empty.
\end{obs}

\begin{obs}\label{obs4}
If $A$ is an arc of the disk diagram, then by Observation \ref{obs1}, we see
that the regions immediately to one side of $A$ alternate singular and
cone regions, while the other side of $A$ alternate singular and flat
vertices. The latter sequence of regions gives us a stalling path in
$\F^{(1)}$. Thus the vertices associated to this sequence of regions
all lie in the same parallel set.
\end{obs}

\subsection{Existence of 1-shells or 2-shells (by the seashore).} 
Following McCammond and Wise, we define a boundary region of a disk to
be an {\it $i$-shell} if it has $i$ internal edges. In the example
above, the core of the diagram has four 2-shells. Since the core of
the diagram is a C(4)-T(4) complex, we can apply Greendlinger's lemma. 
We give a slightly different statement suited to
our needs and we for the sake of completeness, we include a proof here
(for deeper delvings into such results see \cite{MW}).

\begin{lemma}\label{Shells}
Let $D$ be a disk tiled by polygons $P_1,\cdots,P_k$ in such a way
that every interior vertex is incident to at least 4 polygons and so
that each polygon has at least 4 sides. Then one of the following
holds.
\begin{enumerate}
\item $k=1$ (i.e. there is one 0-shell)
\item there are at least two 1-shells
\item there is at least one 1-shell and at least two 2-shells
\item there are at least four 2-shells
\end{enumerate}
\end{lemma}

\begin{proof}
A {\it corner} of $P_i$ is an unordered pair of adjacent edges. Thus
an $n$-gon has $n$ corners (assuming $n\geq 3$). We say
that a corner of $P_i$ is contained in $\partial D$ if the interiors of
the corresponding two edges do not intersect any other $P_j$'s.
Denote by $n_i$ the number of sides of $P_i$ and by $C_i$ the number
of corners of $P_i$ contained in $\partial D$.

View each $P_i$ as a square complex with $n_i$ squares so that
curvature is concentrated at one interior vertex. The excess angle at
this vertex is
$(\pi/2)n_i-2\pi$. Thus the total interior excess is at least
$$(\pi/2)\sum n_i-2k\pi$$ (it might be larger if there are interior
verices incident to $>4$ polygons).

The boundary deficit is $\pi/2$ for each corner, i.e. at most 
$$(\pi/2)\sum C_i$$ (it might be smaller if there are boundary
vertices incident to $>2$ polygons). Recall that Gauss-Bonnet says:

$$(\mbox{boundary deficit})-(\mbox{interior excess})=2\pi$$

Thus

$$(\pi/2)\sum C_i - (\pi/2)\sum n_i + 2k\pi \geq 2\pi$$
and hence

$$\sum C_i \geq 4 + \sum (n_i-4)$$

The statement now follows quickly. To each $P_i$ assign the score of
$C_i-n_i+4$ points. The inequality says that the sum of all points is
$\ge 4$. If $P_i$ is assigned 4 points, it is a 0-shell. If it is assigned 2
points it is a 1-shell, and if it is assigned 1 point it is a
2-shell. (An interior triangle would get 1 point but we are assuming
there are no triangles. All other polygons get a nonpositive number of
points.)
\end{proof}

\begin{remark}
Suppose that there are precisely two 1-shells and no 2-shells. Then
the polygons can be renumbered so that $P_1$ and $P_k$ are 1-shells
and $P_i$ and $P_j$ share an edge iff $|i-j|\leq 1$. In other words,
the polygons form a {\it ladder}. 
\end{remark}

\subsection{Shells, short cuts and taut cycles}

A  cycle $\alpha$ is said to have an {\it i-cut} if there are flat vertices 
$v, w $ on  $\alpha$ such that the coarse length of any path along the cycle 
between $v$ and $w$ is greater than $i$, but for 
which there exists full edge path of length $i$ in $\F^{(1)}$ joining $v$ 
and $w$. A cycle is said to be {\it taut} if there exists no $i$-cut with $i\leq 2$. 

Now suppose that $D$ is a disk diagram for $\alpha$ and let $D^\prime$ be 
its core. If $D^\prime$ is not a single cell, then  by Lemma \ref{Shells}, 
$D^\prime$ has a 1-shell or a 2-shell. As we see in the figure below, we 
then obtain a 1-cut or a 2-cut for $\alpha$.

\begin{center}
\begin{figure}[here]
\begin{center}
\includegraphics{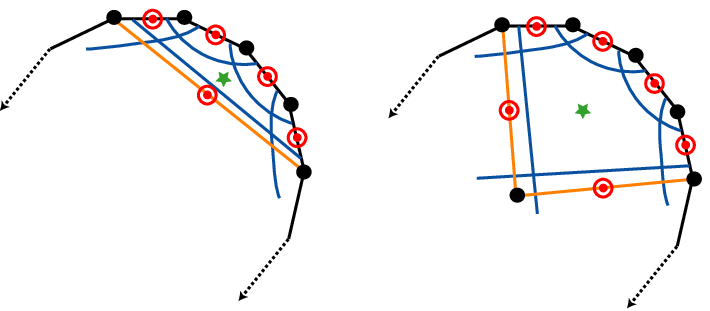}
\end{center}
\caption{1-cuts and 2-cuts arising from shells.}
\end{figure}
\end{center}

Thus we obtain the following lemma.

\begin{lemma}(Diagrams for taut cycles) Suppose that $\alpha$ is a taut cycle. Then the diagram for $\alpha$ has a core consisting of a single cell.
\label{TautDiagram}
\end{lemma}

%
%SECTION:RIGIDITY OF TAUT CYCLES
%

\section{Rigidity of taut cycles}
\label{secrigiditytautcycles}

The aim of this section is to prove the the quasi-isometric
preservation of taut cycles.  Recall that if $\phi:X_\Gamma\to
X_\Gamma$ is a quasi-isometry, Theorem \ref{StandardFlats} tells us
that standard flats are coarsely preserved by $\phi$. Since no two
standard flats are coarsely equivalent, $\phi$ induces a bijection
$\ps$ on the flat vertex set of $\F$. Moreover, since coarse equivalence is
preserved by quasi-isometries, we have that $D^\infty$ is preserved.

\begin{theorem}[Taut cycle rigidity]
Suppose that $\phi$ and $\ps$ are as above. Then $\ps$ carries taut
cycles to taut cycles. In particular, full edges that lie along taut
cycles are carried to full edges.
\label{TautCycleRigidity}
\end{theorem}

We first prove a lemma that tells us that ``quasi'' cuts of length at
most 3 actually give 1-cuts or 2-cuts.

\begin{lemma}[quasi-cuts yield cuts] Let   $\alpha$ be a cycle.
Suppose that there exist non-adjacent flat vertices $v$ and $w$ along
$\alpha$, subdividing $\alpha$ into two paths $\alpha_1$ and
$\alpha_2$. Suppose further that $v$ and $w$ are joined by a path $\beta$ so that
\begin{enumerate}
\item $length_\coarse(\beta)\leq 3$
\item $v$ and $w$ are the only flat vertices of $\alpha$ which lie in $\beta$. 
\end{enumerate}
Then $\alpha$ is not taut. 
\label{QuasiCuts}
\end{lemma}

\begin{proof}

We give the argument when $length_\coarse(v,w)=3$. The argument when
$length_\coarse(v,w)<3$ is similar and indeed simpler. So we suppose
that $\beta$ is broken up by two flat vertices $p$ and $q$ into three
stalling subpaths $\beta_1, \beta_2$ and $\beta_3$ as in Figure
\ref{DoubleIceCream}

\begin{figure}[here]
\begin{center}
\includegraphics{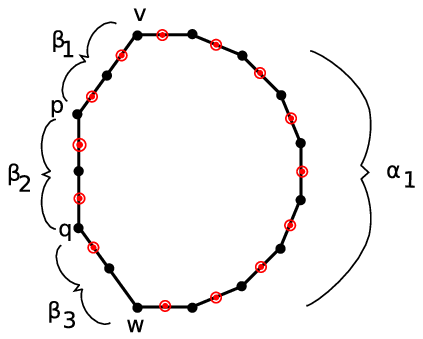}
\end{center}
\caption{A dual disk diagram for a stalling cut.}
\label{DoubleIceCream}
\end{figure}

We let $D$ denote a dual disk diagram for the loop
$\gamma=\beta\cup\alpha_1$. Note that by assumption, this loop is indeed a cycle. We will apply Lemma \ref{Shells} to $D$.
Now we must consider some cases.

{\bf \underline{Case 1: $\beta_1$  and $\beta_3$ are both
    longer than 1.}} Since there are no corners along stalling sections
    of $\beta$, it follows that $D^\prime$ is not a single cell so we
    are in cases 2,3, or 4 of Lemma \ref{Shells}. Now a shell gives at
    least two consecutive corners along $\partial D$. It follows that
    if a shell produces corners along $\beta$, it is either a sequence
    of corners that begins at $v$ or $w$ and continues into $\alpha_1$
    or the sequence of corners is simply the pair of corners at $p$
    and $q$ (in this case, $\beta_2$ is of length 1).

{\bf There exist 1-shells (case 2 or 3 of Lemma \ref{Shells}).} 
Let us suppose we have a 1-shell in $D^\prime$. Now a 1-shell produces
a sequence of corners of length at least three, so it cannot occur at
$p$ and $q$.  Also, both of the endpoints of the 1-cut produced by
this shell cannot be along the non-stalling section of $D$, for
otherwise $\alpha$ would not be taut. Thus, one of the endpoints of
the 1-cut must be along the stalling sections $\beta_1$ or
$\beta_3$. But since these stalling section can only have a corner at
$v$ and $w$, the 1-cut provides a 2-cut between $v$ or $w$ and some
vertex $z$ in the non-stalling section (see Figure
\ref{OneCutAtCorner}). Thus $\alpha$ is not taut and we are done.
\begin{figure}[here]
\begin{center}
\includegraphics{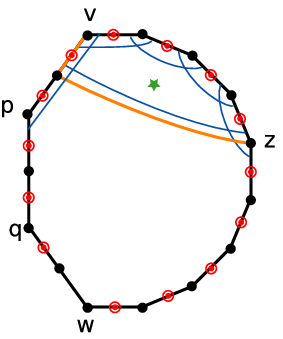}
\end{center}
\caption{A 1-cut yielding a 2-cut for $\alpha$.}
\label{OneCutAtCorner}
\end{figure}

{\bf There exist no 1-shells (case 4 of Lemma \ref{Shells}).} 
In this case there are at least four 2-shells. Since only two of these
can produce corners at the vertices $v$ and $w$, and only one can
produce corners at $p$ and $q$, we must have a 2-shell which produces
a 2-cut in the non-stalling section, namely along $\alpha_1$. This
produces a 2-cut for $\alpha$ and we are done.

{\bf \underline{Case 2: Only one of $\beta_1, \beta_3$ has
    length 1.}}  Assume without loss of generality that $\beta_3$ has
length 1, and $\beta_1$ is longer.  We proceed as in Case 1. If there
exist 1-shells, there are at least two 1-shells or a 1-shell and two
2-shells. It follows that one of them cannot be along
$\beta_1\cup\beta_2$.  As before, these shells cannot occur entirely
along $\alpha_1$. Thus we are in the case in which there are only two
$i$-shells, $i\leq 2$; that is, we have two 1-shells. One of these
must have its first corner at $v$ and we get a 2-cut as in Figure
\ref{OneCutAtCorner}.

If we have only 2-shells, then there are 4 of them, and hence one of
them must occur entirely along the  $\alpha_1$,
producing a 2-cut for $\alpha_1$, contradicting the tautness of
$\alpha$.

{\bf \underline{Case 3: both $\beta_1$ and $\beta_3$ have length 1.}}
Now if $D$ does not have a corner at both $v$ and $w$, then we may
proceed as before. So suppose that there are corners both at $v$ and
$w$. We now return to our original loop $\alpha=\alpha_1\cup
\alpha_2$. Let $v_1$ be the flat vertex immediately adjacent to $v$
along $\alpha_1$ and let $v_2$ be the flat vertex immediately adjacent
to $v$ along $\alpha_2$. Let $v_3$ be the vertex along $\beta$
immediately adjacent to $v$. We define $w_1,w_2$ and $w_3$ similarly
as neighboring vertices of $w$.  Now since there is a corner at $v$,
it follows that the path $[v_3,v,v_1]$ is non-stalling. Thus, the path
$[v_3,v,v_2]$ is a stalling path. Similarly, the path $[w_3,w,w_2]$ is
a stalling path. ( See Figure \ref{NotBothLegal}.)

\begin{figure}[here]
\begin{center}
\includegraphics{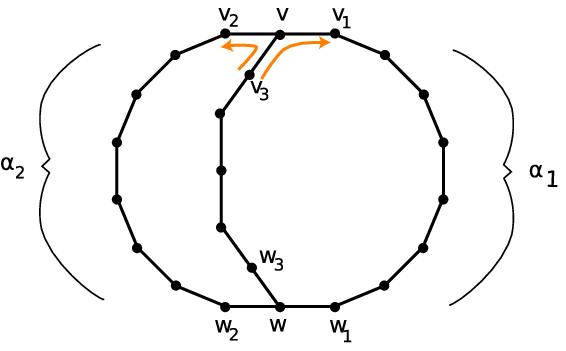}
\end{center}
\caption{One of the paths of length 2 running through $v$ must be stalling.}
\label{NotBothLegal}
\end{figure}

We thus let $\alpha_2^\prime $ be the subarc of $\alpha_2$ spanned by
$v_2$ and $w_2$, and let $\beta^\prime$ be
$\beta\cup[v,v_2]\cup[w,w_2]$. We then consider a dual disk diagram
for $\alpha_2\prime\cup\beta^\prime$ and we are back in Case 1 with
$\alpha_1$ replaced by $\alpha_2^\prime$ and $\beta$ replace by
$\beta^\prime$.
\end{proof}

\bigskip\bigskip

\begin{proof}[Proof of Theorem \ref{TautCycleRigidity}]
Suppose that $v_1,..,v_n$ are the flat vertices along a taut cycle
$\alpha$. Let $w_i=\ps(v_i)$. We know that $D^\infty(w_i,w_{i+1})=1$,
so that we can choose a full edge geodesic $\beta_i$ joining $w_i$ to
$w_{i+1}$, such that flat vertices along it lie in the same
parallelism class. We then join the $\beta_i$'s together to get a
closed loop $\beta$. We claim that this loop is a cycle. First of all,
by construction, the $\beta_i$'s are embedded, and at the $w_i$'s we
have turns, so that there is no backtracking. Secondly, if distinct
$\beta_i$'s met along a vertex which was not one of the endpoints,
then we could apply Lemma \ref{QuasiCuts} to conclude that $\alpha$ is
not taut. So now consider a dual disk diagram $D$ for for $\beta$ and
as usual let $D^\prime$ denote the core of $D$. We claim that
$D^\prime$ consists of a single cell.

 If $D^\prime$ has more than one cell, we would then have an arc $a$
in $D$ which separates $D$ into two regions, each of which contains a
1-shell or 2-shell of $D^\prime$.  Now by Observation \ref{obs4},
there would be a stalling path joining non-adjacent vertices $z_1$ and
$z_2$ of $\beta$. Now if we have $z_1=w_i$ and $z_2=w_j$, for some
$i,j$, then we have that $D^\infty(v_i,v_j)=1$ and we apply Lemma
\ref{QuasiCuts} to conclude that $\alpha$ was not taut.  Now if
$z_1,z_2\not\in\{w_1,...,w_n\}$, then let $w_i$ be a closest such
vertex to $z_1$ and $w_j$ be a closest such vertex to $z_2$. Letting
$z_1^\prime=\ps^{-1}(z_1)$ and $z_2^\prime=\ps^{-1}(z_2)$. We see that
the path $\gamma=[v_i, z_1^\prime, z_2^\prime, v_j]$ satisfies the
conditions of Lemma \ref{QuasiCuts}. We can apply Lemma
\ref{QuasiCuts} in a similar manner when only one of the $z_i$'s is in
$\{w_1,...,w_n\}$.

We thus have that $D^\prime$ consists of a single cell. This means
that in fact the $w_i$ define a closed edge loop $\beta$. Now we
wish to show that this loop is taut. Suppose not. Then we have a 1-cut
or a 2-cut of $\beta$. Pulling such a cut back via $\ps^{-1}$ gives a
quasi-cut as in Lemma \ref{QuasiCuts} thus implying that $\alpha$ is
not taut, a contradiction.
\end{proof}

\section{$\F$-rigidity}
\label{secfrigidity}

We next need to argue that if $F_1,F_2$ are two standard flats in $\F$
that intersect in a line, then their images
$\phi_\#(F_1),\phi_\#(F_2)$ intersect in a line as well. {\it A
priori}, we only know that they intersect coarsely in a line,
i.e. that they belong to the same parallel set. If there is a taut
cycle that crosses both $F_1$ and $F_2$ then the fact that
$\phi_\#(F_1)\cap \phi_\#(F_2)$ is a line follows from taut cycle
rigidity. However, there are pairs of flats where there are no such
taut cycles. This is of course the case if the graph contains valence
1 vertices, but there are no such vertices by our standing assumption
on $\Gamma$. For a more subtle failure see the
example below.

\subsection{Lemmas about graphs}
\label{seclemmasaboutgraphs}

\begin{definition}
Let $\gamma$ be an embedded cycle  in the defining graph $\Gamma$. An {\it
$i$-shortcut} is an edge-path $\beta$ in $\Gamma$ of length $i$ whose
endpoints are in $\gamma$ and whose distance in $\gamma$ is $>i$. 
\end{definition}

\begin{definition}
A cycle $\gamma$ is {\it tight} if it does not admit any 1-shortcuts
nor any 2-shortcuts.  
\end{definition}

\begin{remark}
If $\gamma$ is a cycle in $\Gamma$, we can lift it to a cycle
$\hat\gamma$ in the barycentric subdivision
$\Gamma'\subset \F$. Then $\gamma$ is tight in $\Gamma$ iff
$\hat\gamma$ is taut in $\F$.
\end{remark}

\begin{example}
Let $\Gamma$ be the atomic graph obtained from the 1-skeleton of a
dodecahedron by doubling along one of the pentagons representing a
face. Then every tight cycle is contained in one of the two copies of
the dodecahedron, so that a pair of edges that share a common vertex
but are not contained in the same dodecahedron give rise to flats that
intersect but are not both crossed by a tight cycle.
\end{example}

\begin{definition}
Let $v$ be a vertex of $\Gamma$ and $Lk(v)$ the link of $v$ in
$\Gamma$. The {\it Whitehead graph} at $v$ is the graph $Wh(v)$ whose
vertex set is $Lk(v)$ and $e,e'\in Lk(v)$ (represented by edges with
initial vertex $v$) are connected by an edge iff there is a tight cycle
in $\Gamma$ that enters $v$ along $\overline e$ and leaves along
$e'$. 
\end{definition}

In the above example, the Whitehead graph at every valence 3 vertex is
the complete graph (i.e. a triangle), while at the valence 4 vertices
it is obtained from the complete graph by removing an edge (i.e. it is
a square with one diagonal).

\begin{lemma}\label{whitehead}
$Wh(v)$ is connected iff $v$ is not a cut vertex.
\end{lemma}

\begin{proof}
It is clear that $Wh(v)$ is disconnected if $v$ is a cut vertex. For
the converse, suppose $Wh(v)$ is disconnected, and let
$e_1,\cdots,e_p\in Lk(v)$ be the vertices of one of the components of
$Wh(v)$, and let $f_1,\cdots,f_q$ be all the other vertices of
$Lk(v)$. Suppose $v$ is not a cut vertex. Then there are embedded cycles in
$\Gamma$ that enter $v$ along some $\overline e_i$ and leave along
some $f_j$. Let $C$ be a shortest such cycle. We claim that $C$ is tight.  To
see this, suppose $\be$ is an $i$-shortcut for $C$ for $i\leq 2$.  Using the
fact that $\Ga$ has no $4$-cycle, it follows that one may 
replace a subpath of $C$ with $\be$ to obtain  a shorter cycle $C'$ which 
contains some $\overline e_i$ and some $f_j$.  Thus
 some $e_i$ is connected to some $f_j$ in $Wh(v)$, a
contradiction. 
\end{proof}

\begin{lemma}
An atomic graph contains no cut vertices, or separating closed edges.
\end{lemma}
\proof
Let $Y\subset \Ga$ be either a vertex or a closed edge, and suppose $v\in Y$
is a vertex.  Note that if $C\subset \Ga\setminus Y$ is a connected component,
then $C$ cannot be contained in the the closed star of $v$, because it would
then lie in a single edge of $\Ga$, which clearly contradicts the assumption
that every vertex has valence at least $2$.
\qed

\begin{lemma}\label{no cuts}
Suppose the edges of an atomic graph $\Gamma$ are colored in three
colors, black, white, and gray, and the following holds:
\begin{enumerate}
\item any
  two gray edges share a vertex, and
\item the edges of any tight cycle are either all black or gray, or
  they are all white or gray.
\end{enumerate}
Then the edges of $\Gamma$ are either all black or gray, or else they
are all white or gray.
\end{lemma}

\begin{proof}
Suppose not. Thus there is an edge $b$ colored black, and
there is an edge $w$ colored white. All gray edges have
a vertex $v$ in common (by assumption (1) and because there are no
triangles). Let $G$ be the union of all (closed) gray edges. We
now claim that $G$ does not separate $\Gamma$. Indeed, take $x,y\in
\Gamma - G$. If $x,y\notin St(v)$ then $x,y$ can be joined by a path
missing $St(v)$ and hence missing $G$. Say $x\in St(v)-G$. If $x$ is a
vertex, then it is adjacent to $v$ and it is incident to an edge $e$
different from $[x,v]$ (since otherwise $[x,v]$ would be a separating
edge). Thus $x$ can be joined by a path in the complement of $G$ to
the complement of $St(v)$. If $x$ is not a vertex, then it is an
interior point of some edge $[x',v]$ and it can again be joined by a
path in the complement of $G$, first to $x'$ and then to the
complement of $St(v)$ as before. Similar considerations apply to $y$
and our claim is proved.

Now join $b$ and $w$ by a path missing $G$. Thus all edges along
this path are black or white, and there is a vertex $z$ along this path
where color changes from black to white. Therefore all edges incident to
$z$ are black or white and at least one is black and at least one is
white. By Lemma \ref{whitehead} there is a tight cycle passing through
$z$ and at $z$ crossing one white and one black edge. But this
contradicts our assumption (2).
\end{proof}

\begin{lemma}\label{edges to isos}
Let $\Gamma,\Gamma'$ be graphs with no vertices of valence $<2$ and
  no cycles of length $<4$. Suppose $F:{\mathcal
  E}(\Gamma)\to {\mathcal E}(\Gamma')$ is a bijection of the sets
  of
  edges of $\Gamma$, $\Gamma'$ such that
\begin{itemize}
\item if $e_1,e_2\in {\mathcal E}(\Gamma)$ share a vertex then
  $F(e_1),F(e_2)\in {\mathcal E}(\Gamma')$ share a vertex.
\end{itemize}
Then there is a graph isomorphism $\phi:\Gamma\to\Gamma'$ that induces
$F$.
\end{lemma}

\begin{proof}
Let $v$ be a vertex of $\Gamma$. Consider the set of all edges
$e_1,\cdots,e_k$ in $\Gamma$ that contain $v$. By our assumptions on $\Gamma$,
$k\geq 2$ and any two of the edges $F(e_1),\cdots,F(e_k)$ share a
vertex. Since $\Gamma'$ has no cycles of length $<4$ it follows that
there is a unique vertex $w\in\Gamma'$ contained in all
$F(e_i)$. Define $\phi(v):=w$. By construction, if $v_1,v_2$ span an edge
then $\phi(v_1),\phi(v_2)$ span an edge. Thus $\phi:\Gamma\to\Gamma'$
is a simplicial map. Reversing the roles of $\Gamma,\Gamma'$ provides
a simplicial map $\phi':\Gamma'\to\Gamma$ and it easy to see that
$\phi\phi'$ and $\phi'\phi$ are identity.
\end{proof}

\subsection{$\F$-rigidity}
\label{subsecfrigidity}

The following is the main theorem in the paper. It says that the flat
space $\F$ is a quasi-isometry invariant of the RAAG $G$. The notation should be
self-explanatory, e.g. $\F_1=\F(\Gamma_1)$.

\begin{theorem}\label{frigidity}
Let $\Gamma_1,\Gamma_2$ be two atomic graphs. Let $f:G_1\to G_2$ be a quasi-isometry
between the associated RAAGs. Then there is a label-preserving isometry
$\phi:\F_1\to\F_2$ between the associated flat spaces such that for each
standard flat (i.e. a flat vertex) $v$
we have that $f(v)$ is contained in a Hausdorff neighborhood of the
flat $\phi(v)$.

In particular, $\phi$ takes each cone in $\F_1$ isometrically to a cone in
$\F_2$, and 
$\Gamma_1\cong\Gamma_2$. 
\end{theorem}

\begin{proof}
By Theorem \ref{StandardFlats} we have a function
$\phi:\F_1^\flat\to\F_2^\flat$ defined on the set of flat vertices and
satisfying the statement about the Hausdorff neighborhood. This map is
a bijection. We now extend this map to an isometry.

{\bf Step 1.} Let $C(\Gamma_1)\subset\F_1$ be the fundamental domain (see
Proposition \ref{action}(2)). We claim that $\phi$ takes all flat
vertices in $\Gamma_1\subset C(\Gamma_1)$ into a translate of the
fundamental domain $C(\Gamma_2)\subset \F_2$. By Theorem
\ref{TautCycleRigidity} we know that $\phi$ sends all flat vertices
along a taut cycle in $\Gamma_1\subset C(\Gamma_1)$ into a single cone
(and this cone is unique, see Observation \ref{obs cones}). Since
every flat vertex in $\Gamma_1$ lies along a taut cycle it suffices to
prove that there is a cone that contains the images of the flat
vertices along any taut cycle in $\Gamma_1\subset C(\Gamma_1)$.

So suppose this is not true, and that for two different taut cycles in
$\Gamma_1$ the images of the flat vertices are contained in different
cones $C',C''\subset \F_2$. Choose a parallel set $P\subset \F_2$
that separates between $C'$ and $C''$ (see Lemma
\ref{parallel sets}) so that we can write $\F_2=X'\cup X''$ with
$C'\subset X'$, $C''\subset X''$ and $X'\cap X''= P$. Now
color the edges (i.e. the flat vertices) of $\Gamma_1$ into black,
white or gray according to whether they are mapped into $X'$, $X''$
or $X'\cap X''$. Note that 
any two gray edges share a vertex (the corresponding flats coarsely
intersect in a line; for edges of $\Gamma_1$ this happens only when the
edges share a vertex). Now
by Lemma \ref{no cuts} we see that all flat vertices in $\Gamma_1$ are
mapped to either $X'$ or to $X''$ contradicting our choices.

Thus, after composing with an element of $G_2$, we may assume that
$\phi$ maps the flat vertices in $C(\Gamma_1)$ to flat vertices in
$C(\Gamma_2)$. By considering the inverse map, $\phi$ restricts to a bijection
${\mathcal E}(\Gamma_1)\to {\mathcal E}(\Gamma_2)$ between the sets of
edges of $\Gamma_1$ and $\Gamma_2$.

This analysis can be applied to any translate of $C(\Gamma_1)$. As a
result of Step 1, we can extend $\phi$ to the cone vertices so that
it induces a bijection between the sets of cone vertices and any
adjacent pair of a flat and a cone vertex maps to an adjacent pair.

{\bf Step 2.} Now consider $\phi$ restricted to the ``base of the
cone'', that is, the map ${\mathcal E}(\Gamma_1)\to {\mathcal
  E}(\Gamma_2)$.
Apply Lemma \ref{edges to isos} to deduce that
$\phi$ extends to an isomorphism $\phi:\Gamma_1\to\Gamma_2$. Adjacent
edges map to adjacent edges since only in this situation the
corresponding flats coarsely intersect in a line, so the lemma
applies. 

{\bf Step 3.} It remains to extend $\phi$ to the singular vertices. Let $v$
be a singular vertex, say in $C(\Gamma_1)$. Then there is a unique way to
define $\phi(v)$ so that $\phi$ preserves adjacency with the flat
vertices in $C(\Gamma_1)$ (this is Step 2). We need to verify that when
we regard $v$ as a vertex in a different cone we obtain the same
$\phi(v)$. There are $\Z$ cones that contain $v$ (see Proposition
\ref{action}) and all these cones also contain all flat vertices
adjacent to $v$, so the claim follows.
\end{proof}

\section{The proofs of corollaries \ref{corqiiffisom}, \ref{cormostow}, 
and \ref{coroutisom}}\label{proofs of corollaries}

{\em Proof of Corollary \ref{corqiiffisom}.}  Suppose $\Ga$ and $\Ga'$
are atomic graphs, and $\phi:K\ra K'$ is a quasi-isometry.  Let 
$\phi_0:V\ra V'$ be the bijection provided by  Theorem \ref{thmmain1}.
Pick $p\in K$ and let $p'\defeq \phi_0(p)$.  Then $\phi_0$ establishes
a bijection between the standard geodesics (respectively standard flats)
passing through $p$ and 
the standard geodesics (respectively standard flats) passing through $p'$.
This induces an isomorphism $\Ga\ra \Ga'$ of the defining graphs.
\qed

\bigskip
{\em Proof of Corollary \ref{cormostow}.}
Let $\psi:G\ra G'$ be an isomorphism.  Then we obtain an isometric action
$G\acts K'$, which is discrete and cocompact.  Hence by the fundamental
lemma of geometric group theory, there is a $G$-equivariant
quasi-isometry  
$$
\phi:K\lra K'.
$$
Let $\phi_0:V\lra V'$ be the bijection furnished by Theorem \ref{thmmain1}.

Consider a standard flat $F=\al\times\be\subset K$.  By Theorem \ref{thmmain1}, there
is a standard flat $F'=\al'\times\be'\subset K'$ such that $\phi_0$ maps $F\cap V$
bijectively to $F'\cap V'$, and respects the product structures.
Since $\phi_0\restr_{F\cap V}$ is equivariant with respect to the stabilizer of 
$F$, it follows that the induced mappings $(\al\cap V)\ra(\al'\cap V')$ and 
$(\be\cap V)\ra (\be'\cap V')$
are also $\stab(F)$-equivariant.  Since $\stab(F)$ acts transitively
on $\al\cap V$, $\be\cap V$, $\al'\cap V'$, and $\be'\cap V'$, the map
$\phi_0\restr_{F\cap V}$ is the restriction of an isometry $F\ra F'$.  
Hence $\phi_0$ is the restriction of an isometry $K\ra K'$.
\qed

\bigskip
{\em Proof of Corollary \ref{coroutisom}.}
Let $\bar\psi:\bar K\ra \bar K'$ be a homotopy equivalence.
Then $\bar\psi$ induces an isomorphism $G\simeq G'$ and
we may lift $\bar\psi$ to a quasi-isometry
$$
\psi:K\lra K'
$$
which is $G$-equivariant (where we identify $G$ and $G'$ using the
isomorphism above). By  Corollary \ref{cormostow} there is a
$G$-equivariant isometry $\phi:K\ra K'$ at bounded distance from 
$\psi$, and this descends to an isometry 
$$
\bar\phi:\bar K\lra \bar K'
$$
which is homotopic to $\bar\psi$.

\qed

\section{Further implications of Theorem \ref{thmmain1}}
\label{secfurtherimplications}

At first sight, one might think that the map
$\phi_0:V\ra V'$ in the conclusion of Theorem \ref{thmmain1} 
must be the restriction of an isometry, since it preserves so
much structure; however, this turns out not to be the case.
In this section we single out part of the structure of the 
bijection $\phi_0$ which efficiently distinguishes between
quasi-isometries, namely we associate with each parallel set
$P\subset K$ a biLipschitz homeomorphism of a copy of the 
integers.  In the next section we will see that any biLipschitz
homeomorphism can arise this way.

\bigskip
Let $\phi:K\ra K'$ be a quasi-isometry, where $K$ and $K'$
are associated with atomic RAAG's, and let $\phi_0:V\ra V'$
be the map of Theorem \ref{thmmain1}.  

Let $\P$ and $\P'$ denote the collections of maximal standard product
subcomplexes in $K$ and $K'$, respectively.
Since $K$ and $K'$ are atomic, each $P\in \P$ or $P'\in \P'$ is the parallel set for 
a standard geodesic.   Therefore $P$ splits isometrically as a product
of complexes $P=\R\times T$.  

\begin{definition}
\label{defzp}
For each parallel set $P\in \P$, we define $\R_P$ to be the 
 $\R$-factor in the splitting
$$
P=\R\times T,
$$
and let $\Z_P$ to be the set of vertices of $\R_P$,
equipped with the induced metric.
\end{definition}

\begin{lemma}
\label{lemparallelsetpreservation}
For every $P\in \P$ there is a $P'\in \P'$ such that $\phi_0$ maps
$V\cap P$ bijectively to $V'\cap P'$, preserving the product structure.
\end{lemma}
\proof
Let $\ga\subset P$ be a standard geodesic parallel to the $\R$-factor
of $P$.  Then Theorem \ref{thmmain1} implies that $V\cap \ga$
is mapped bijectively by $\phi_0$ to $V'\cap \ga'$, for some 
standard geodesic $\ga'\subset K'$.  Since any two geodesics
$\ga_1'$, $\ga_2'$ obtained this way are parallel, it follows
that $\phi_0(P\cap V)$ is contained in a parallel set.
Applying the same reasoning to the inverse implies that $\phi_0(V\cap P)=V'\cap P'$
for some $P'\in \P'$.

Since $\phi_0$ preserves standard geodesics, it follows that
$\phi_0\restr_{V\cap P}$ preserves product structure.
\qed

\bigskip
By abuse of notation we use $\phi$ to denote the induced bijection
$\P\ra\P'$ given by Lemma \ref{lemparallelsetpreservation}.
By the lemma, for each $P\in \P$, we obtain a bijective quasi-isometry
 $\phi_{\Z_P}:\Z_P\ra \Z_{\phi(P)}$, where the quasi-isometry constants are 
controlled by those of $\phi$.

\begin{lemma}
\label{lemphip}
\begin{enumerate}[(1)]
\item If $\phi$ is an isometry, then $\phi_{\Z_P}:\Z_P\ra \Z_{\phi(P)}$ is an isometry
for every $P\in \P$.
\item
If $\phi$ is induced by an element of the commensurator $\comm(G)$, 
then for every $P\in \P$, the map $\phi_{\Z_P}:\Z_P\ra\Z_{P'}$ is equivariant
with respect to cocompact isometric actions on $\Z_P$ and $\Z_{P'}$.
\end{enumerate}
\end{lemma}
\proof
Assertion (1) is immediate.

There is an isomorphism  $\al:H\ra H'$ where $H\subset G$ and $H'\subset G'$
are finite index subgroups, such that $\phi:K\ra K'$ is $H$-equivariant, where
$H$ acts freely, cocompactly, and isometrically on $K'$ via the isomorphism $\al$.
Then $\phi_0:V\ra V'$ is also $H$-equivariant.  Hence if $h\in H$ stabilizes
$P\in \P$, then it also stabilizes the parallel set $\phi(P)\in\P'$.  
In other words, the map $\phi_0$ restricts
to a $\stab(P)\cap H$-equivariant mapping $V\cap P\ra V'\cap \phi(P)$.  Since 
$\stab(P)\cap H$ acts cocompactly on $P$,  assertion (2) follows.
\qed

\section{Quasi-isometric flexibility and the proof of Theorem~\ref{thmflexibility}}
\label{secflexibility}

{\em The homomorphism $\aut(G)\ra \comm(G)$ is injective.}
Suppose $\al\in \aut(G)$.  By Corollary \ref{cormostow}
there is an isometry $\phi:K\ra K$ which induces $\al$, i.e.
we identify $G$ with a subgroup of $\isom(K)$.  If 
$\al\in \ker(\aut(G)\ra\comm(G))$, then $\phi$ commutes
with a finite index subgroup of $G$, and therefore has
bounded displacement,
$$
\sup_{p\in K}\;d(\phi p,p)<\infty.
$$
If follows that $\phi$ maps each standard flat to itself.
Since the intersection of the standard flats passing through 
a vertex $p\in K$ is precisely $p$, it follows that $\phi_0$
fixes every vertex, and is therefore the identity map.

\bigskip
{\em The homomorphism $\comm(G)\ra \qi(G)$ is injective.}
Suppose $\al\in \comm(G)$.  Then $\al$ can be represented
by an isomorphism $f:G_1\ra G_2$, where the $G_i$'s are finite
index subgroups of $G$.  Therefore there is a quasi-isometry
$\phi:K\ra K$ which is $G_1$-equivariant, where we identify
$G_1$ with $G_2$ via $f$, and use the corresponding action on the
second copy of $K$.  If $\al\in \ker(\comm(G)\ra \qi(G))$,
then $\phi$ is at bounded distance from the identity.
Letting $\phi_0:V\ra V$ be the bijection given by Theorem \ref{thmmain1},
we may argue as in the preceding paragraph to conclude
that $\phi_0=\id_V$.  It follows that $\al=\id$.

\bigskip
{\em $\aut(G)$ has infinite index in $\comm(G)$.}
Pick a vertex $v\in \Ga$, and a positive integer $k$.
Let $\al:G\ra\Z_k$ be the homomorphism which 
sends the generator $v\in G$ to $1\in \Z_k$,
and the other generators to $0\in \Z_k$.  
Let $G_k\subset G$ be the kernel of $\al$, and $\bar K_k$
be the $k$-fold cover of $\bar K$ given by $\bar K_k\defeq K/G_k$.
One may describe $\bar K_k$ as follows.  Begin with $\bar K(\Star v)$,
which is a product $S^1\times B_j$, where $B_j$ is a bouquet
of $j$-circles, where $j$ is the number of vertices adjacent
to $v$.  To obtain $\bar K$ from $\bar K(\Star v)$, one 
glues on a copy of $\bar K(\Ga_1)$, where $\Ga_1\subset\Ga$
is the graph obtained by deleting $v$ and the  edges incident to $v$.
To obtain $\bar K_k$, one  first passes to the $k$-fold cyclic
cover of $\bar K(\Star v)$, and then glues on $k$-copies of
$\bar K(\Ga_1)$.

Let $\Ga_k$ be the graph obtained by taking $k$ copies of
$\Ga$, and gluing them together along the $k$ copies of $\Star(v)$.
In fact $\bar K(\Ga_k)$ is  homotopy equivalent to $\bar K_k$;
to see this, map $\bar K_k$ to $\bar K( G_k)$ by taking $(k-1)$ of
the copies of  $[0,1]\times B_j\subset \bar K_k$ and collapsing them
to copies of $B_j$ (by collapsing the interval factors).
In particular, $G_k$ is isomorphic to the RAAG $G(\hat G_k)$.
Since $\hat \Ga_k$ is not atomic, this shows that the atomic
condition is not commensurability invariant among RAAG's.

Note that the permutation group $S_k$ of the set $\Z_k$ acts
isometrically on $\bar K(\Ga_k)$ by permuting the copies of $\bar
K(\Ga_1)$, and hence we get a homomorphism $S_k\ra \out(G_k)\ra
\comm(G)/\inn(G)$.  For each element $\al\in S_k$, we may lift the
corresponding homotopy equivalence $\psi:\bar K_k\ra \bar K_k$ to a
quasi-isometry $\phi:K\ra K$ which preserves a parallel set $P\subset
K$ covering $\bar K(\Star v)$.  Moreover, the induced map
$\phi_{\Z_P}:\Z_P\ra \Z_P$ will be equivariant with respect to the
action of $k\Z$ on $\Z_P$ by translations, and descends to the
permutation of $\Z_k=\Z/k\Z$ corresponding to $\al$.

Now consider the collection $\C$ of elements $\phi\in \comm(G)$ 
obtained this way, as $k$ varies over the positive integers, and
$\al$ varies over the permutation group of $\Z_k$.  If $\aut(G)$
had finite index in $\comm(G)$, we could finite a finite collection
$f_1,\ldots,f_i\in \C$ such that for each  $\phi\in \C$ there is a
$\psi\in \isom(K)$ such that $\phi\circ \psi=f_j$ for some $j\in \{1,\ldots,i\}$.
This means that  $\psi^{-1}(P)=f_j^{-1}(P)=:\bar P$, and  that 
$$
\phi_{\Z_P}    \circ\psi_{\Z_{\bar P}}:\Z_{\bar P}\ra \Z_{\bar P}
$$
agrees with $(f_j){\Z_{\bar P}}:\Z_{\bar P}\ra \Z_{\bar P}$.
However, by part 1 of Lemma \ref{lemphip}, $\psi_{\Z_{\bar P}}$
is an isometry.  This clearly contradicts the fact that $\phi$
can come from any permutation $\al$  of $\Z_k$, for any $k$.

\bigskip
\begin{remark}
\label{remnogeometricaction}
Let $H$ denote the group of isometries of $K(\Ga_k)$ covering
the permutation action $S_k\acts \bar K(\Ga_k)$.  Then there is 
a short exact sequence 
$$
1\lra G(\Ga_k)\simeq G_k\lra H\lra S_k\lra 1.
$$
Thus $H$ is commensurable with $G(\Ga)$.  However, there is no
geometric action of $H$ on $K(\Ga)$.  This can be deduced by examining
the subgroup of $H$ which stabilizes a parallel set $P$, 
and observing that the induced action on $\Z_P$ is not
conjugate to an isometric action.
\end{remark}

\bigskip
{\em $\comm(G)$ has infinite index in $\qi(G)$.}
The construction is similar to the proof that $[\comm(G):\aut(G)]=\infty$.

Pick $v\in \Ga$, and look at the infinite cyclic cover $\bar K_\infty$
corresponding
to the homomorphism $G\ra \Z$ which sends $v$ to $1$ and the other
generators to $0$.   Then $\bar K_\infty$ can be obtained from 
$\R\times B_j$ by gluing on infinitely many copies of $\bar K(\Ga_1)$.
As in the preceding paragraphs, we may produce homotopy equivalences
by ``permuting the copies of $\bar K(\Ga_1)$''.  By lifting these homotopy 
equivalences we may obtain quasi-isometries
$\phi:K\ra K$ which preserve a parallel set $P$
covering $\bar K(\Star v)$, and hence obtain a bijective
quasi-isometry $\phi_{\Z_P}:\Z_P\ra \Z_P$.  It is not hard to see that
any bijective quasi-isometry $\Z_P\ra \Z_P$ may be obtained in this
way.  

If $[\qi(G):\comm(G)]$ were finite, there would only be finitely
many possibilities for the $\phi_{\Z_P}$'s up to  pre-composition
by maps of the form $\psi_P$, where $\psi$ comes from an element
of the commensurator.  In view of Lemma \ref{lemphip}, this is clearly
not the case.

\bibliography{artin}

\begin{thebibliography}{KMLNR80}

\bibitem[Ahl02]{ahlin}
A.~R. Ahlin.
\newblock The large scale geometry of products of trees.
\newblock {\em Geom. Dedicata}, 92:179--184, 2002.
\newblock Dedicated to John Stallings on the occasion of his 65th birthday.

\bibitem[Bas72]{bass}
H.~Bass.
\newblock The degree of polynomial growth of finitely generated nilpotent
  groups.
\newblock {\em Proc. London Math. Soc. (3)}, 25:603--614, 1972.

\bibitem[BKS07]{quasiflats}
M.~Bestvina, B.~Kleiner, and M.~Sageev.
\newblock Quasiflats in ${C}{A}{T}(0)$ $2$-complexes.
\newblock Preprint, 2007.

\bibitem[BM00]{burgermozes}
M.~Burger and S.~Mozes.
\newblock Lattices in product of trees.
\newblock {\em Inst. Hautes \'Etudes Sci. Publ. Math.}, (92):151--194 (2001),
  2000.

\bibitem[BN06]{behrstockneumann}
J.~Behrstock and W.~Neumann.
\newblock Quasi-isometric classification of graph manifold groups.
\newblock math.GT/0604042, 2006.

\bibitem[CCV]{ccv}
R.~Charney, J.~Crisp, and K.~Vogtmann.
\newblock Automorphisms of two-dimensional right-angled artin groups.
\newblock math.GR/0610980.

\bibitem[CD95]{charneydavis}
Ruth Charney and Michael~W. Davis.
\newblock The {$K(\pi,1)$}-problem for hyperplane complements associated to
  infinite reflection groups.
\newblock {\em J. Amer. Math. Soc.}, 8(3):597--627, 1995.

\bibitem[Cha]{charney}
Ruth Charney.
\newblock An introduction to right angled {A}rtin groups.
\newblock arXiv, math.GR/0610668.

\bibitem[Dav98]{buildingsarecat0}
Michael~W. Davis.
\newblock Buildings are {${\rm CAT}(0)$}.
\newblock In {\em Geometry and cohomology in group theory (Durham, 1994)},
  volume 252 of {\em London Math. Soc. Lecture Note Ser.}, pages 108--123.
  Cambridge Univ. Press, Cambridge, 1998.

\bibitem[Dro87]{MR891135}
Carl Droms.
\newblock Isomorphisms of graph groups.
\newblock {\em Proc. Amer. Math. Soc.}, 100(3):407--408, 1987.

\bibitem[Dun85]{dunwoody}
M.~J. Dunwoody.
\newblock The accessibility of finitely presented groups.
\newblock {\em Invent. Math.}, 81(3):449--457, 1985.

\bibitem[Gro81]{polygrowth}
M.~Gromov.
\newblock Groups of polynomial growth and expanding maps.
\newblock {\em Inst. Hautes \'Etudes Sci. Publ. Math.}, (53):53--73, 1981.

\bibitem[Gro87]{hypgps}
M.~Gromov.
\newblock Hyperbolic groups.
\newblock In {\em ``Essays in group theory''}, pages 75--263. Springer, New
  York, 1987.

\bibitem[Gro93]{MR1253544}
M.~Gromov.
\newblock Asymptotic invariants of infinite groups.
\newblock In {\em Geometric group theory, Vol.\ 2 (Sussex, 1991)}, volume 182
  of {\em London Math. Soc. Lecture Note Ser.}, pages 1--295. Cambridge Univ.
  Press, Cambridge, 1993.

\bibitem[HM96]{harlandermeinert}
Jens Harlander and Holger Meinert.
\newblock Higher generation subgroup sets and the virtual cohomological
  dimension of graph products of finite groups.
\newblock {\em J. London Math. Soc. (2)}, 53(1):99--117, 1996.

\bibitem[J{\'S}01]{MR1875609}
Tadeusz Januszkiewicz and Jacek {\'S}wi{\c{a}}tkowski.
\newblock Commensurability of graph products.
\newblock {\em Algebr. Geom. Topol.}, 1:587--603 (electronic), 2001.

\bibitem[KKL98]{kapovichkleinerleeb}
M.~Kapovich, B.~Kleiner, and B.~Leeb.
\newblock Quasi-isometries and the de {R}ham decomposition.
\newblock {\em Topology}, 37(6):1193--1211, 1998.

\bibitem[KL97]{kapovichleeb}
M.~Kapovich and B.~Leeb.
\newblock Quasi-isometries preserve the geometric decomposition of {H}aken
  manifolds.
\newblock {\em Invent. Math.}, 128(2):393--416, 1997.

\bibitem[KL07]{klinduced}
B.~Kleiner and B.~Leeb.
\newblock Induced quasi-actions: a remark.
\newblock preprint, 2007.

\bibitem[KMLNR80]{kimetal}
K.~H. Kim, L.~Makar-Limanov, J.~Neggers, and F.~W. Roush.
\newblock Graph algebras.
\newblock {\em J. Algebra}, 64(1):46--51, 1980.

\bibitem[KPS73]{MR0349850}
A.~Karrass, A.~Pietrowski, and D.~Solitar.
\newblock Finite and infinite cyclic extensions of free groups.
\newblock {\em J. Austral. Math. Soc.}, 16:458--466, 1973.
\newblock Collection of articles dedicated to the memory of Hanna Neumann, IV.

\bibitem[Lau95]{laurence}
M.~R. Laurence.
\newblock A generating set for the automorphism group of a graph group.
\newblock {\em J. London Math. Soc. (2)}, 52(2):318--334, 1995.

\bibitem[MSW03]{moshersageevwhyte}
L.~Mosher, M.~Sageev, and K.~Whyte.
\newblock Quasi-actions on trees. {I}. {B}ounded valence.
\newblock {\em Ann. of Math. (2)}, 158(1):115--164, 2003.

\bibitem[MW02]{MW}
Jonathan~P. McCammond and Daniel~T. Wise.
\newblock Fans and ladders in small cancellation theory.
\newblock {\em Proc. London Math. Soc. (3)}, 84(3):599--644, 2002.

\bibitem[Sag95]{sageevthesis}
Michah Sageev.
\newblock Ends of group pairs and non-positively curved cube complexes.
\newblock {\em Proc. London Math. Soc. (3)}, 71(3):585--617, 1995.

\bibitem[Ser89]{MR1023285}
Herman Servatius.
\newblock Automorphisms of graph groups.
\newblock {\em J. Algebra}, 126(1):34--60, 1989.

\bibitem[Sta68]{stallings}
J.~R. Stallings.
\newblock On torsion-free groups with infinitely many ends.
\newblock {\em Ann. of Math. (2)}, 88:312--334, 1968.

\bibitem[Wis96]{wise}
D.~Wise.
\newblock Non-positively curved squared complexes, aperiodic tilings, and
  non-residually finite groups.
\newblock Princeton thesis, 1996.

\end{thebibliography}
\bibliographystyle{alpha}

\end{document}